\documentclass{siamart}

\usepackage{xcolor}
\usepackage{mathrsfs} 

\usepackage{mathtools}
\usepackage{comment}
\usepackage{caption}
\usepackage{subcaption}
\usepackage{amsmath}
\usepackage{amsfonts}
\usepackage{amssymb}
\usepackage{algorithm}
\usepackage{algpseudocode}
\usepackage{enumerate}

\newcommand{\minimize}[1]{{\displaystyle\min_{#1}}}

\newcommand{\Footnote}[1]{}

\newcommand{\red}[1]{{\color{red}#1\color{black}}}
\newcommand{\blue}[1]{{\color{blue}#1\color{black}}}

\newcommand{\diag}{\operatorname*{diag}}
\newcommand{\interior}{\operatorname{int}}
\newcommand{\boundary}{\operatorname{bnd}}
\newcommand{\bd}{\operatorname{bd}}
\newcommand{\na}{\nabla}
\newcommand{\RR}{\mathbb{R}}

\newcommand{\CCC}{\mathcal{C}}
\newcommand{\AAA}{\mathcal{A}}
\newcommand{\RRR}{\mathcal{R}}

\newcommand{\KKK}{\mathcal{K}}

\newcommand{\III}{\mathcal{I}}
\newcommand{\EEE}{\mathcal{E}}
\newcommand{\DDD}{\mathcal{D}}
\newcommand{\BBB}{\mathcal{B}}
\newcommand{\FFF}{\mathcal{F}}
\newcommand{\JJJ}{\mathcal{J}}

\newcommand{\NNN}{\mathcal{N}}
\newcommand{\YYY}{\mathcal{Y}}
\newcommand{\lam}{\lambda}

\newcommand{\xbar}{\skew{2.8}\bar x}

\newcommand{\setB}{\mathcal{B}}
\newcommand{\setC}{\mathcal{C}}

\newcommand{\setJ}{\mathcal{J}}
\newcommand{\setK}{\mathcal{K}}

\newcommand{\setT}{\mathcal{T}}

\newcommand{\subject}{\text{s.t.} \ }
\newcommand{\abs}[1]{|#1|}
\newcommand{\norm}[1]{\|#1\|}
\newcommand{\xstar}{x^*}

\newcommand{\bmat}[1]{\begin{bmatrix}#1\end{bmatrix}} 
\newtheorem{remark}[theorem]{Remark}
\newtheorem{assumption}{Assumption}

\title{A Quadratically Convergent Sequential Programming Method for Second-Order Cone Programs Capable of Warm Starts}

\author{Xinyi Luo\thanks{Department of Industrial Engineering and Management Sciences, Northwestern University.  This author was partially supported by National Science Foundation grant DMS-2012410.
        E-mail: \email{xinyiluo2023@u.northwestern.edu}}
   \and Andreas W\"achter\thanks{Department of Industrial Engineering and Management Sciences, Northwestern University.  This author was partially supported by National Science Foundation grant DMS-2012410.
        E-mail: \email{andreas.waechter@northwestern.edu}}}

\begin{document}

\maketitle

\begin{abstract}
We propose a new method for linear second-order cone programs.  
It is based on the sequential quadratic programming framework for nonlinear programming. In contrast to interior point methods, it can capitalize on the warm-start capabilities of active-set quadratic programming subproblem solvers and achieve a local quadratic rate of convergence.

In order to overcome the non-differentiability or singularity observed in nonlinear formulations of the conic constraints, the subproblems approximate the cones with polyhedral outer approximations that are refined throughout the iterations.
For nondegenerate instances, the algorithm implicitly identifies the set of cones for which the optimal solution lies at the extreme points.  As a consequence, the final steps are identical to regular sequential quadratic programming steps for a differentiable nonlinear optimization problem, yielding local quadratic convergence.

We prove the global and local convergence guarantees of the method and present numerical experiments that confirm that the method can take advantage of good starting points and can achieve higher accuracy compared to a state-of-the-art interior point solver.

%
%

\end{abstract}

\begin{keywords}
  nonlinear optimization, second-order cone programming, sequential quadratic programming
\end{keywords}

\begin{AMS}
90C15, 90C30, 90C55 
\end{AMS}

\section{Introduction}

We are interested in the solution of second-order cone programs (SOCPs) of the form
\begin{subequations}\label{eq:socp}
\begin{align}
\min_{x\in\RR^n} \ & c^Tx \\ 
\text{s.t.} \ & Ax\leq b, \label{eq:socp_2}\\
& x_j \in \KKK_j \qquad j\in\JJJ :=\{1,\ldots,p\}, \label{eq:socp_3}
\end{align}
\end{subequations}
where $c\in\RR^n$, $b\in\RR^m$, $A\in\RR^{m\times n}$, and $x_j$ is a subvector of $x$ of dimension $n_j$ with index set $\III_j\subseteq\{1,\ldots,n\}$.  We assume that the sets $\III_j$ are disjoint. The set $\KKK_j$ is the second-order cone of dimension $n_j$, i.e.,
\begin{equation}\label{eq:def_soc}
\KKK_j := \{y \in\RR^{n_j} :  \|\bar y\| \leq y_0\},
\end{equation}
where the vector $y$ is partitioned into  $y=(y_0, \bar y^T)^T$ with $\bar y=(\bar y_1,\ldots,\bar y_{n_j-1})^T$.
These problems arise in a number of important applications \cite{alizadeh2003second,ben2001polyhedral,molzahn2019survey,Tseng2007}

\begin{sloppypar}
Currently, most of the commercial software for solving SOCPs implements interior-point algorithms which utilize a barrier function for second-order cones, see, e.g.\ \cite{gurobi,cplex,mosek}. Interior-point methods have well-established global and local convergence guarantees \cite{potra2000interior} and are able to solve large-scale instances, but they cannot take as much of an advantage of a good estimate of the optimal solution as it would be desirable in many situations.
For example, in certain applications, such as online optimal control, the same optimization problem has to be solved over and over again, with slightly modified data. In such a case, the optimal solution of one problem provides a good approximation for the new instance. Having a solver that is capable of ``warm-starts'', i.e., utilizing this knowledge, can be essential when many similar problems have to be solved in a small amount of time.
\end{sloppypar}

For some problem classes, including linear programs (LPs), quadratic programs (QPs), or nonlinear programming (NLP), active-set methods offer suitable alternatives to interior-point methods.  They explicitly identify the set of constraints that are active (binding) at the optimal solution.  When these methods are started from a guess of the active set that is close to the optimal one, they often converge rapidly in a small number of iterations.  An example of this is the simplex method for LPs.  Its warm-start capabilities are indispensable for efficient branch-and-bound algorithms for mixed-integer linear programs. 

Active-set methods for LPs, QPs, or NLPs are also known to outperform interior-point algorithms for problems that are not too large \cite{ctx37679402130002441,hei2008numerical,7074396}.  Similarly, active-set methods might be preferable when there are a large number of inequality constraints among which only a few are active, since an interior-point method is designed to consider all inequality constraints in every iteration and consequently solves large linear systems, whereas an active set method can ignore all inactive inequality constraints and encounters potentially much smaller linear systems.

Our goal is to propose an active-set alternative to the interior-point method in the context of SOCP that might provide similar benefits.
We introduce a new sequential quadratic programming (SQP) algorithm that, in contrast to interior-point algorithms for SOCPs, has favorable warm-starting capabilities because it can utilize active-set QP solvers.
We prove that it is globally convergent, i.e., all limit points of the generated iterates are optimal solutions under mild assumptions, and that it enjoys a quadratic convergence rate for non-degenerate instances.  
Our preliminary numerical experiments demonstrate that these theoretical properties are indeed observed in practice.
They also show that the algorithm is able in some cases to compute a solution to a higher degree of precision than interior point methods.
This is expected, again in analogy to the context of LPs, QPs, and NLPs, since an interior point method terminates at a small, but nonzero value of the barrier parameter that cannot be made smaller than some threshold (typically $10^{-6}$ or $10^{-8}$) because the arising linear systems become highly ill-conditioned.  In contrast, in the final iteration of the active-set method, the linear systems solved correspond directly to the optimality conditions, without any perturbation introduced by a barrier parameter, and are only as degenerate as the optimal solution of the problem.

%
%


The paper is structured as follows.  Section~\ref{sec:prelim} reviews the sequential quadratic programming method and the optimality conditions of SOCPs.  Section~\ref{sec:alg} describes the algorithm, which is based on an outer approximation of the conic constraints.  Section~\ref{sec:theory} establishes the global and local convergence properties of the method, and numerical experiments are reported in Section~\ref{sec:numerical}.  Concluding remarks are offered in Section~\ref{sec:conclusion}.


\subsection{Related work}
While a large number of interior-point algorithms for SOCP have been proposed, including some that have been implemented in efficient optimization packages \cite{gurobi,cplex,mosek},
there are only very few approaches for solving SOCPs with an active-set framework.
The method proposed by Goldberg and Leyffer \cite{goldberg2015active} is a two-phase algorithm that combines a projected-gradient method with equality-constrained SQP.  However, it is limited to instances that have only conic constraints \eqref{eq:socp_3} and no additional linear constraints \eqref{eq:socp_2}.
Hayashi et al.\ \cite{hayashi2016simplex} propose a simplex-type method, where they reformulate the SOCP as a linear semi-infinite program to handle the fact that these instances have infinitely many extreme points.  The resulting dual-simplex exchange method shows promising practical behavior.  However, in contrast to the method proposed here, the authors conjecture that their method has only an R-linear local convergence rate.
Zhadan \cite{zhadan2021variant} proposes a similar simplex-type method.
Another advantage of the method presented in this paper is that the pivoting algorithm does not need to be designed and implemented from scratch.  Instead, it can leverage existing implementations of active-set QP solvers, in particular the efficient handling of linear systems.

    %

The proposed algorithm relies on polyhedral outer approximations based on well-known cutting planes for SOCPs.
%
%
For instance, the methods for mixed-integer SOCP by Drewes and Ulbrich \cite{drewes2012subgradient} and Coey et al.~\cite{coey2020outer} use
these cutting planes to build LP relaxations of the branch-and-bound subproblems.
%
%
%
We note that an LP-based cutting plane algorithm for SOCP could be seen as an active-set method, but it is only linearly convergent.
As pointed out in \cite{diehl2006loss}, it is crucial to consider the curvature of the conic constraint in the subproblem objective to achieve fast convergence.

    %

The term ``SQP method for SOCP'' has also been used in the literature to refer to methods for solving nonlinear SOCPs \cite{diehl2006loss,kato2007sqp,okuno2015sl1qp,zhang2012trust}.  However, in contrast to the method here, in these approaches, the subproblems themselves are SOCPs \eqref{eq:socp}  
and include the linearization of the nonlinear objective and constraints.
It will be interesting to explore extensions of the proposed method to nonlinear SOCPs in which feasibility is achieved asymptotically not only for the nonlinear constraints but also for the conic constraints.

\subsection{Notation}

For two vectors $x,y\in\RR^n$, we denote with $x\circ y$ their component-wise product, and the condition $x\perp y$ stands for $x^Ty=0$.  For $x\in\RR^n$, we define $[x]^+$ as the vector with entries $\max\{x_i,0\}$. 
%
%
We denote by $\norm{\cdot}$, $\|\cdot\|_1$, $\|\cdot\|_\infty$  the Euclidean norm, the $\ell_1$-norm, and the $\ell_\infty$-norm, respectively.
For a cone $\KKK_j$, $e_{ji}\in\RR^{n_j}$ is the canonical basis vector with $1$ in the element corresponding to $x_{ji}$ for $i\in\{0,\ldots,n_j-1\}$,
and $\interior (\setK_j)$ and $\bd(\setK_j)$ denote the cone's interior and boundary, respectively.

\section{Preliminaries}\label{sec:prelim}


The NLP reformulation of the SOCP is introduced in Section~\ref{s:reformulation}.  We review in Section~\ref{sec:localSQP} the local convergence properties of the SQP method and in Section~\ref{sec:penaltySQP} the penalty function as a means to promote convergence from any starting point.
In Section~\ref{s:optcondSOCP}, we briefly state the optimality conditions and our assumption for the SOCP~\eqref{eq:socp}.

\subsection{Reformulation as a smooth optimization problem}\label{s:reformulation}

The definition of the second-order cone in \eqref{eq:def_soc} suggests that the conic constraint \eqref{eq:socp_3} can be replaced by the nonlinear constraint
\begin{equation*}
  r_j(x_j) := \|\bar x_j\| - x_{j0} \leq 0
\end{equation*}
without changing the set of feasible points.
Consequently, \eqref{eq:socp} is equivalent to
\begin{subequations}\label{eq:nlp_r}
    \begin{align}
\min_{x\in \RR^n} \ & c^Tx\\
\text{s.t.}\ & Ax\leq b, \\
 & r_j(x_j)  \leq 0, \qquad j\in\JJJ. \label{eq:nlp_r_3}
\end{align}
\end{subequations}
%

Unfortunately, \eqref{eq:nlp_r} cannot be solved directly with standard gradient-based algorithms for nonlinear optimization, such as SQP methods.  
The reason is that $r_j$ is not differentiable whenever $\bar x_j=0$.  
This is particularly problematic when the optimal solution $x^*$ of the SOCP lies at the extreme point of a cone,  $x^*_j=0\in\KKK_j$.  
In that case, the Karush-Kuhn-Tucker (KKT) necessary optimality conditions for the NLP formulation, which are expressed in terms of derivatives, cannot be satisfied.  Therefore, any optimization algorithm that seeks KKT points cannot succeed.
As a remedy, differentiable approximations of $r_j$ have been proposed in the past; see, for example, \cite{vanderbei1998loqo}.
However, high accuracy comes at the price of high curvature, which can make finding the numerical solution of the NLP difficult. 


An alternative equivalent reformulation of the conic constraint is given by
\[
\|\bar x_j\|^2 - x_{j0}^2\leq 0 \text{ and } x_{j0}\geq 0.
\]
In this case, the constraint function is differentiable.  But if $x_j^*=0$, its gradient vanishes, and as a consequence, no constraint qualification applies and the KKT conditions do not hold.  Therefore, again, a gradient-based method cannot be employed.
By using an outer approximation of the cones that is improved in the course of the algorithm, our proposed variation of the SQP method is able to avoid these kinds of degeneracy.

To facilitate the discussion we define a point-wise partition of the cones.
\begin{definition}
Let $x\in\RR^n$. 
\begin{enumerate}
\item  We call a cone $\KKK_j$ \emph{extremal-active} at $x$, if $x_j=0$, and we denote with
$
\EEE(x) = \left\{j\in\JJJ :x_j=0 \right\}
$
the set of extremal-active cones at $x$.
\item We define the set
$
\DDD(x) = \left\{j\in\JJJ : \bar x_j\neq 0\right\}
$
as the set of all cones for which the function $r_j$ is differentiable at $x$.
\item We define the set
$
\NNN(x) = \left\{j\in\JJJ :x_j\neq0 \text{ and } \bar x_j=0\right\}
$
as the set of all cones that are not extremal-active and for which $r_j$ is not differentiable $x$.
\end{enumerate}
\end{definition}

If the set $\EEE(x^*)$ at an optimal solution $x^*$ were known in advance, we could compute $x^*$ as a solution of \eqref{eq:socp} by solving the NLP
\begin{subequations}\label{eq:equiv_nlp}
\begin{align}
\underset{x\in \RR^n}{\min} \ & c^Tx\\
\text{s.t.} \ & Ax\leq b,  \label{eq:equiv_nlp_1}\\
 & r_j(x)  \leq 0,  \quad j\in\DDD(x^*),\label{eq:equiv_nlp_3}\\
 & x_j = 0,  \hspace{0.8cm} j  \in \EEE(x^*)\label{eq:equiv_nlp_4}.
\end{align}
\end{subequations}
The constraints involving the linearization of $r_j$ are imposed only if $r_j$ is differentiable at $x^*$, and variables in cones that are extremal-active at $x^*$ are explicitly fixed to zero.
With this, locally around $x^*$, all functions in \eqref{eq:equiv_nlp} are differentiable and we could apply standard second-order algorithms to achieve fast local convergence.

In \eqref{eq:equiv_nlp}, we omitted the cones in $\NNN(x^*)$.  If $x^*$ is feasible for the SOCP and $j\in\NNN(x^*)$ we have $\bar x^*_j=0$ and $x^*_{j0}>0$, and so $r_j(x^*)<0$.  This implies that the nonlinear constraint \eqref{eq:equiv_nlp_3} for this cone is not active and we can omit it from the problem statement without impacting the optimal solution.


%

\subsection{Local convergence of SQP methods}\label{sec:localSQP}

The proposed algorithm is designed to guide the iterates $x^k$ into the neighborhood of an optimal solution $x^*$.  If the optimal solution is not degenerate and the iterates are sufficiently close to $x^*$, the steps generated by the algorithm are eventually identical to the steps that the SQP method would take for solving the differentiable optimization problem \eqref{eq:equiv_nlp}.
In this section, we review the mechanisms and convergence results of the basic SQP method \cite{nocedal2006numerical}. 

At an iterate $x^k$, the basic SQP method, applied to \eqref{eq:equiv_nlp}, computes a step $d^k$ as an optimal solution to the QP subproblem
\begin{subequations}\label{eq:sqp_subproblem}
\begin{align}
\min_{d\in\RR^n} \ & c^Td + \tfrac12 d^TH^k d\\
\text{s.t.} \ & A(x^k+d) \leq b,\\
& r_j(x_j^k) + \na  r_j(x_j^k)^Td_j  \leq 0,  \hspace{1cm}  j\in\DDD(x^*),\label{eq:sqp_subproblem_2}\\
& x_j^k + d_j=0,\hspace{3.05cm} j\in\EEE(x^*)\label{eq:sqp_subproblem_3}.
\end{align}
\end{subequations}
Here, $H^k$ is the Hessian of the Lagrangian function for \eqref{eq:equiv_nlp}, which in our case is
\begin{equation}\label{eq:sqphessian}
    H^k = \sum_{j\in\DDD(x^*)}\mu^k_j\na^2_{xx}r_j(x_j^k),
\end{equation}
where $\mu_j^k\geq 0$ are estimates of the optimal multipliers for the nonlinear constraint \eqref{eq:equiv_nlp_3},
and where $\na^2_{xx}r_j(x_j)$ is the $n\times n$ block-diagonal matrix with
\begin{equation}\label{eq:rHess}
\na^2 r_j(x_j) = 
\begin{bmatrix}
 0 &  0 \\
 0 & \frac{1}{\|\bar x_j\|}I - \frac{\bar x_j\bar x_j^T}{\|\bar x_j\|^3}
\end{bmatrix}
\end{equation}
in the rows and columns corresponding to $x_j$ for $j\in\JJJ$.
It is easy to see that $\na^2 r_j(x_j)$ is positive semi-definite.
The estimates $\mu_j^k$ are updated based on the optimal multipliers $\hat\mu_j^k\geq 0$ corresponding to \eqref{eq:sqp_subproblem_2}.

\begin{algorithm}[t]
    \caption{Basic SQP Algorithm}\label{alg:SQP}
    \begin{algorithmic}[1]
        \Require Initial iterate $x^0$ and multiplier estimates $\lam^0$, $\mu^0$, and $\eta^0$.
        \For{$k=0,1,2\cdots$}
        \State Compute $H^k$ from \eqref{eq:sqphessian}.
        \State Solve QP \eqref{eq:sqp_subproblem} to get step $d^k$ and multipliers $\hat\lambda^k$, $\hat\mu_j^k$, and $\hat\eta^k$.
        \State Set $x^{k+1}\gets x^k+ d^k$ and $\mu_j^{k+1}\gets\hat\mu_j^k$ for all $j\in\DDD(x^*)$. 
        \label{s:SQPfullstep}
        \EndFor
    \end{algorithmic}
\end{algorithm}

Algorithm~\ref{alg:SQP} formally states the basic SQP method where $\hat\lambda^k\geq 0$ and $\hat\eta^k$ denote the multipliers corresponding to \eqref{eq:equiv_nlp_1} and \eqref{eq:equiv_nlp_4}, respectively.
Because we are only interested in the behavior of the algorithm when $x^k$ is close to $x^*$, we assume here that $\bar x_j^k\neq0$ for all $j\in\DDD(x^*)$ and for all $k$, and hence the gradient and Hessian of $r_j$ can be computed.
Note that the iterates $\hat\lambda^k$ and $\hat\eta^k$ are not explicitly needed in Algorithm~\ref{alg:SQP}, but they are necessary to measure the optimality error and define the primal-dual iterate sequence that is analyzed in Theorem~\ref{thm:sqp}.

A fast rate of convergence can be proven under the following sufficient second-order optimality assumptions \cite{boggs1995sequential}.
\begin{assumption}\label{ass:sqp}
Suppose that $x^*$ is an optimal solution of the NLP \eqref{eq:equiv_nlp} with corresponding KKT multipliers $\lambda^*$, $\mu^*$, and $\eta^*$, satisfying the following properties:
\begin{enumerate}[(i)]
\item\label{ass:strictcompl} Strict complementarity holds;
\item\label{ass:licq} the linear independence constraint qualification (LICQ) holds at $x^*$, i.e., the gradients of the constraints that hold with equality at $x^*$ are linearly independent; 
\item\label{ass:sosc} the projection of the Lagrangian Hessian $H^*=\sum_{j\in\DDD(x^*)}\mu^*_j\na^2_{xx}r_j(x_j^*)$ into the null space of the gradients of the active constraints is positive definite.
\end{enumerate}
\end{assumption}

\begin{sloppypar}
Under these assumptions, the basic SQP algorithm reduces to Newton's method applied to the optimality conditions of \eqref{eq:equiv_nlp} and the following result holds \cite{nocedal2006numerical}.
\begin{theorem}\label{thm:sqp}
Suppose that Assumption~\ref{ass:sqp} holds and that the initial iterate $x^0$ and multipliers $\mu^0$ (used in the Hessian calculation) are sufficiently close to $x^*$ and $\mu^*$, respectively.  Then the iterates $(x^{k+1},\hat\lam^k,\hat\mu^k,\hat\eta^k)$ generated by the basic SQP algorithm, Algorithm \ref{alg:SQP}, converge to $(x^*,\lam^*,\mu^*,\eta^*)$ at a quadratic rate.
\end{theorem}
\end{sloppypar}

\subsection{Penalty function}\label{sec:penaltySQP}
Theorem~\ref{thm:sqp} is a local convergence result.  Practical SQP algorithms include mechanisms that make sure that the iterates eventually reach such a neighborhood, even if the starting point is far away.
To this end, we employ the exact penalty function
\begin{equation}\label{eq:merit}
\varphi(x;\rho) = c^Tx + \rho\sum_{j\in\JJJ}[r_j(x_j)]^+
\end{equation}
in which $\rho>0$ is a penalty parameter.
Note that we define $\varphi$ in terms of all conic constraints $\JJJ$, even though $r_j$ appears in \eqref{eq:equiv_nlp_3} only for $j\in\DDD(x^*)$.   We do this because the proposed algorithm does not know $\DDD(x^*)$ in advance and the violation of all cone constraints needs to be taken into account when the original problem \eqref{eq:socp} is solved.  Nevertheless, in this section, we may safely ignore the terms for $j\not\in\DDD(x^*)$ because for $j\in\EEE(x^*)$  we have $x^k_j=0$ and hence $[r_j(x^k)]^+=0$ for all $k$ due to \eqref{eq:sqp_subproblem_3}, and when $j\in\NNN(x^*)$, we have $r_j(x_j^k)<0$ when $x^k$ is close to $x^*$ since $r_j(x_j^*)<0$.

It can be shown, under suitable assumptions, that the minimizers of $\varphi(\,\cdot\,;\rho)$ over the set defined by the linear constraints \eqref{eq:equiv_nlp_1},
\begin{equation}\label{eq:setX}
X = \{x\in\RR^n : Ax\leq b\},
\end{equation}
coincide with the minimizers of \eqref{eq:equiv_nlp} when $\rho$ is chosen sufficiently large.
Because it is not known upfront how large $\rho$ needs to be, the algorithm uses an estimate, $\rho^k$, in iteration $k$, which might be increased during the course of the algorithm.

To ensure that the iterates eventually reach a minimizer of $\varphi(\,\cdot\,;\rho)$, and therefore a solution of \eqref{eq:equiv_nlp}, we require that the decrease of $\varphi(\,\cdot\,;\rho)$ is at least a fraction of that achieved in the piece-wise linear model of $\varphi(\,\cdot\,;\rho)$ given by
\begin{equation}\label{eq:merit_model}
m^k(x^k+d;\rho) = c^T(x^k+d) + \rho\sum_{j\in\DDD(x^k)}[r_j(x^k_j)+\na r_j(x^k_j)^Td_j]^+,
\end{equation}
constructed at $x^k$.
More precisely, the algorithm accepts a trial point $\hat x^{k+1}=x^k+d$ as a new iterate only if the sufficient decrease condition
\begin{align}\label{eq:armijo}
\varphi(\hat x^{k+1};\rho^k) - \varphi(x^k;\rho^k) \leq &\,c_{\text{dec}}\Big(m^k(x^k+d;\rho^k) - m^k(x^k;\rho^k)\Big)\\
\stackrel{\eqref{eq:merit_model}}{=} &\,c_{\text{dec}} \biggl(c^Td - \rho^k\sum_{j\in\DDD(x^k)}[r_j(x^k_j)]^+\biggr)\nonumber
\end{align}
holds with some fixed constant $c_{\text{dec}}\in(0,1)$.
The trial iterate $\hat x^{k+1}=x^k+d^k$ with $d^k$ computed from \eqref{eq:sqp_subproblem} might not always satisfy this condition.  The proposed algorithm generates a sequence of improved steps of which one is eventually accepted.

However, to apply Theorem~\ref{thm:sqp}, it would be necessary that the algorithm take the original step $d^k$ computed from \eqref{eq:sqp_subproblem}; see Step~\ref{s:SQPfullstep} of Algorithm~\ref{alg:SQP}.
Unfortunately, $\hat x^{k+1}=x^k+d^k$ might not be acceptable even when the iterate $x^k$ is arbitrarily close to a non-degenerate solution $x^*$ satisfying Assumption~\ref{ass:sqp} (a phenomenon called the Maratos effect \cite{maratos1978exact}).
Our remedy is to employ the second-order correction step \cite{fletcher1982second}, $s^k$, which is obtained as an optimal solution of the QP
\begin{subequations}\label{eq:sqp_soc}
\begin{align}
\min_{s\in\RR^n} \ & c^T(d^{k}+s) + \tfrac12 (d^{k}+s)^TH^k (d^{k}+s)\\
\text{s.t.} \ & A(x^k+d^{k}+s) \leq b,\label{eq:sqp_soc_2}\\
& r_j(x_j^k+d_j^{k}) + \na  r_j(x_j^k+d_j^{k})^Ts_j \leq 0,   \qquad j\in\DDD(x^*),\label{eq:sqp_soc_3}\\
& x_j^k + d^{k}_j + s_j =0,\hspace{3.55cm} j\in\EEE(x^*).\label{eq:sqp_soc_4}
\end{align}
\end{subequations}
For later reference, let $\hat\lam^{S,k}$, $\hat\mu^{S,k}$ and $\hat\eta^{S,k}$  denote optimal multiplier vectors corresponding to \eqref{eq:sqp_soc_2}--\eqref{eq:sqp_soc_4}, respectively.
The algorithm accepts the trial point $\hat x^{k+1}=x^k+d^{k}+s^k$ if it yields sufficient decrease \eqref{eq:armijo} with respect to the original SQP step $d=d^k$.
Note that \eqref{eq:sqp_soc} is a variation of the second-order correction that is usually used in SQP methods, for which \eqref{eq:sqp_soc_3} reads
\[
r_j(x_j^k+d_j^{k}) + \na  r_j(x_j^k)^T(d_j^{k}+s_j)  \leq 0,   \qquad j\in\DDD(x^*),
\]
and avoids the evaluation of $\na r_j(x_j^k+d_j^k)$.  In our setting, however, evaluating $\na r_j(x_j^k+d_j^k)$ takes no extra work and \eqref{eq:sqp_soc_3} is equivalent to a supporting hyperplane, see Section~\ref{sec:support}.
%
%
As the following theorem shows (see, e.g., \cite{fletcher1982second} or \cite[Section 15.3.2.3]{conn2000trust}), this procedure computes steps with sufficient decrease \eqref{eq:armijo} and results in quadratic convergence.
\begin{sloppypar}
\begin{theorem}\label{thm:sqp_soc}
Let Assumption~\ref{ass:sqp} hold and assume that the initial iterate $x^0$ and multipliers $\mu^0$ are sufficiently close to $x^*$ and $\mu^*$, respectively. Further suppose that $\rho^k=\rho^{\infty}$ for large $k$ where $\rho^{\infty}>\mu_j^*$ for all $j\in\DDD(x^*)$. 
\begin{enumerate}
\item Consider an algorithm that generates a sequence of iterates by setting $(x^{k+1},\lam^{k+1},\mu^{k+1},\eta^{k+1})=(x^k+d^k,\hat\lam^k,\hat\mu^k,\hat\eta^k)$ or $(x^{k+1},\lam^{k+1},\mu^{k+1},\eta^{k+1})=(x^k+d^k+s^k,\hat\lam^{S,k},\hat\mu^{S,k},\hat\eta^{S,k})$ for all $k=0,1,2,\ldots$.
Then $(x^{k},\lam^k,\mu^{k},\eta^{k})$ converges to $(x^*,\lam^*, \mu^*,\eta^*)$ at a quadratic rate. 
\item Further, for all $k$, either $\hat x^{k+1}=x^k+d^k$ or $\hat x^{k+1}=x^k+d^k+s^k$ satisfies the acceptance criterion \eqref{eq:armijo}.

\end{enumerate}
\end{theorem}
\end{sloppypar}


\subsection{Optimality conditions for SOCP}\label{s:optcondSOCP}
The proposed algorithm aims at finding an optimal solution of the SOCP~\eqref{eq:socp}, or equivalently, values of the primal variables, $x^*\in\RR^n$, and the dual variables, $\lam^*\in\RR^m$ and $z_j^*\in\RR^{n_j}$ for $j\in\JJJ$, that satisfy the necessary and sufficient optimality conditions \cite[Theorem 16]{alizadeh2003second}
\begin{subequations}\label{eq:opt_socp}
\begin{align}
c + A^T\lam^* - z^* & = 0,\label{eq:opt_socp_1}\\
A x^* - b \leq 0 & \perp\lam^* \geq 0,  \label{eq:opt_socp_2} \\
\KKK_j \ni x_j^* & \perp z_j^* \in \KKK_j, \qquad j\in\JJJ. \label{eq:opt_socp_3}
\end{align}
\end{subequations}
A thorough discussion of SOCPs is given in the comprehensive review by Alizadeh and Goldfarb \cite{alizadeh2003second}. 
The authors consider the formulation in which the linear constraints \eqref{eq:socp_2} are equality constraints, but the results in \cite{alizadeh2003second} can be easily extended to inequalities.

The primal-dual solution $(x^*,\lam^*,z^*)$ is unique under the following assumption.
\begin{assumption}\label{ass:socpnondegen}
$(x^*,\lam^*,z^*)$ is a non-degenerate primal-dual solution of the SOCP \eqref{eq:socp} at which strict complementarity holds. 
\end{assumption}

The definition of non-degeneracy for SOCP is somewhat involved and we refer the reader to \cite[Theorem 21]{alizadeh2003second}.
Strict complementarity holds if $x_j^*+z_j^*\in\interior(\KKK_j)$ and implies that:
(i) $x^*_j\in\interior(K_j) \Longrightarrow z^*_j=0$; (ii) $z^*_j\in\interior(K_j) \Longrightarrow x^*_j=0$; (iii) $x^*_j\in\bd(K_j)\setminus\{0\} \Longleftrightarrow z^*_j\in\bd(K_j)\setminus\{0\}$; and (iv) not both $x^*_j$ and $z_j^*$ are zero.

\section{Algorithm}\label{sec:alg}

The proposed algorithm solves the NLP formulation \eqref{eq:nlp_r} using a variation of the SQP method.  Since the functional formulation of the cone constraints \eqref{eq:nlp_r_3} might not be differentiable at all iterates or at an optimal solution, the cones are approximated by a polyhedral outer approximation using supporting hyperplanes.

The approximation is done so that the method implicitly identifies the constraints that are extremal-active at an optimal solution $x^*$, i.e., $\EEE(x^*)=\EEE(x^k)$ for large $k$.
More precisely, we will show that close to a non-degenerate optimal solution, the steps generated by the proposed algorithm are identical to those computed by the QP subproblem \eqref{eq:sqp_subproblem} for the basic SQP algorithm for solving \eqref{eq:equiv_nlp}.  Consequently, fast local quadratic convergence is achieved, as discussed in Section~\ref{sec:localSQP}.

\subsection{Supporting hyperplanes}\label{sec:support}

In the following, consider a particular cone $\KKK_j$ and let $\YYY_j$ be a finite subset of 
$\{y_j\in\RR^{n_j}:\bar y_j\neq0, y_{j0}\geq 0\}$.  
We define the cone
\begin{equation}\label{eq:CCC_j}
\CCC_j(\YYY_j) = \left\{ x_j\in\RR^{n_j} : x_{j0}\geq 0 \text{ and }\na  r_j(y_j)^Tx_j\leq 0 \text{ for all }y_j\in\YYY_j\right\}
\end{equation}
generated by the points in $\YYY_j$.
For each $x_j\in\KKK_j$ we have $r_j(x_j)\leq0$, and using
\begin{equation}\label{eq:na_r}
\na r_j(x_j) = \biggl(-1, \frac{\bar x_j^T}{\|\bar x_j\|}\biggr)^T,
\end{equation}
we obtain for any $y_j\in\YYY_j$ that
\[
\na r_j(y_j)^Tx_j = \frac{1}{\|\bar y_j\|}\bar y_j^T\bar x_j - x_{j0} \leq \frac{1}{\|\bar y_j\|}\|\bar y_j\|\|\bar x_j \|- x_{j0} = r_j(x_j)\leq 0.
\]
Therefore $\CCC_j(\YYY_j)\supseteq\KKK_j$. 
Also, for $y_j\in\YYY_j$, consider 
$
x_j = (1,\bar y_j^T/\|\bar y_j\|)^T.
$
Then
\[
\na r_j(y_j)^Tx_j = \frac{\bar y_j^T}{\|\bar y_j\|}\frac{\bar y_j}{\|\bar y_j\|} - 1 = 1- 1= 0,
\]
and also
$
r_j(x_j) = \|\bar x_j\| - x_{j0} = {\bar y_j}/{\|\bar y_j\|} -1 = 0.
$
Hence $x_j\in\CCC_j(\YYY_j)\cap\KKK_j$.  
Therefore, for any $y_j\in\YYY_j$, the inequality
\begin{equation}\label{eq:hyperplane}
	\na  r_j(y_j)^Tx_j\leq 0
\end{equation}
defines a hyperplane that supports $\KKK_j$ at $(1, {\bar y_j}/{\norm{\bar y_j}})$.  In summary, $\CCC_j(\YYY_j)$ is a polyhedral outer approximation of $\KKK_j$, defined by supporting hyperplanes.

In addition, writing $\YYY_j=\{y_{j,1},\ldots,y_{j,m}\}$,
we also define the cone
\begin{equation}\label{eq:Cdual}
\CCC_j^\circ(\YYY_j) := \left\{ - \sum_{l=1}^{m}\sigma_{j,l}\na r_j(y_{j,l}) + \eta_j e_{j0} : \sigma_j\in\RR^{m}_+, \eta_j \geq 0\right\}.
\end{equation}
For all $x_j\in\CCC_j(\YYY_j)$ and $z_j=- \sum_{l=1}^{m}\sigma_{j,l}\na r_j(y_{j,l}) + \eta_j e_{j0} \in \CCC_j^\circ(\YYY_j)$, we have
\[
x_j^Tz_j = - \sum_{l=1}^{m}\sigma_{j,l}\na r_j(y_{j,l})^Tx_j + \eta_j x_{j0}\geq 0
\]
because $\na r_j(y_{j,l})^Tx_j\leq 0$ and $x_{j0}\geq 0$ from the definition of $\CCC_j(\YYY_j)$.
Therefore $\CCC_j^\circ(\YYY_j)$ is included in the dual of the cone $\CCC_j(\YYY_j)$.

Now define $R=[-\na r_j(y_{j,1}),\ldots,-\na r_j(y_{j,m}),e_{j0}]$ and let $z_j\in\RR^{n_j}$ be in the dual of $\CCC_j(\YYY_j)$.  Since this implies that $x_j^Tz_j\geq 0$ for all $x\in\CCC_j(\YYY_j)=\{\RR^{n_j}:R^Tx_j\geq0\}$, Farkas' lemma yields that $z_j=R\cdot(\sigma^T,\eta)^T$ for some $\sigma_j\in\RR^{m}_+$ and $\eta_j \geq 0$, i.e., $z_j\in\CCC_j^\circ(\YYY_j)$.

Overall we proved that $\CCC_j^\circ(\YYY_j)$ defined in \eqref{eq:Cdual} is the dual of $\CCC_j(\YYY_j)$, and since $\CCC_j(\YYY_j)\supseteq\KKK_j$, this implies $\CCC^\circ_j(\YYY_j)\subseteq\KKK_j$.



\subsection{QP subproblem}

In each iteration, at an iterate $x^k$, the proposed algorithm computes a step $d^k$ as an optimal solution of the subproblem
\begin{subequations}\label{eq:qp_Y}
\begin{align}
\min_{d\in\RR^n} \ & c^Td + \tfrac12 d^TH^k d\\
\text{s.t.} \ & A(x^k+d) \leq b, \label{eq:qp_Y_2}\\
& r_j(x_j^k) + \na  r_j(x_j^k)^Td_j  \leq 0,   \qquad j\in\DDD(x^k), \label{eq:qp_Y_3}\\
& x_j^k + d_j \in \CCC_j(\YYY^k_j),\hspace{12.75ex} j\in\JJJ.\label{eq:qp_Y_5}
\end{align}
\end{subequations}
Here, $H^k$ is a positive semi-definite matrix that captures the curvature of the nonlinear constraint \eqref{eq:nlp_r_3}, and for each cone, $\YYY^k_j$ is the set of hyperplane-generating points that have been accumulated up to this iteration.
From \eqref{eq:CCC_j}, we see that \eqref{eq:qp_Y_5} can be replaced by linear constraints.  Consequently, \eqref{eq:qp_Y} is a QP and can be solved as such.

\begin{algorithm}[t]
    \caption{Preliminary SQP Algorithm}\label{alg:basic_primal}
    \begin{algorithmic}[1]
        \Require Initial iterate $x^0$ and sets $\YYY_j^0$ for $j\in\JJJ$.
        \For{$k=0,1,2\cdots$}
        \State Choose $H^k$.
        \State Solve subproblem \eqref{eq:qp_Y} to get step $d^k$. 
        \State Set $x^{k+1}\gets x^k+ d^k$. 
        \State Set $\YYY_j^{k+1}\gets \YYY_{pr,j}^+(\YYY^k_j,x^k_j)$ for $j\in\JJJ$. \label{s:Ypr} 
        \EndFor
    \end{algorithmic}
\end{algorithm}

Algorithm~\ref{alg:basic_primal} describes a preliminary version of the proposed SQP method based on this subproblem.
Observe that the linearization \eqref{eq:qp_Y_3} can be rewritten as
\begin{align*}
0 & \geq r_j(x_j^k) + \na  r_j(x_j^k)^Td_j = \|\bar x^k_j\| - x^k_{j0} - d_{j0} + \frac{(\bar x^k_j)^T\bar d_j}{\|\bar x^k_j\|} \\ 
   & = \frac{1}{\|\bar x^k_j\|}(\bar x^k_j)^T(\bar x^k_j+\bar d_j) - (x^k_{j0}+d_{j0}) 
= 
\na r_j(x^k_j)^T(x^k_j+d_j) 
\end{align*}
and is equivalent to the hyperplane constraint generated at $x_j^k$.
Consequently, if $x^k_j\not\in\KKK_j$, then $r_j(x^k_j)>0$ and \eqref{eq:qp_Y_3} acts as a cutting plane that excludes $x^k_j$.
Using the update rule
\begin{equation}\label{eq:YYYplus_pr}
\YYY^+_{pr,j}(\YYY_j,x_j) =
 \begin{cases}
  \YYY_j\cup \{x_j\} & \text{ if } \bar x_j\neq 0 \text{ and } r_j(x_j)>0, \\
  \YYY_j & \text{ otherwise,}
 \end{cases}
\end{equation}
in Step~\ref{s:Ypr} makes sure that $x^k_j$ is excluded in all future iterations.

In our algorithm, we initialize $\YYY_j^0$ so that 
\begin{equation}\label{eq:Y0}
\YYY_j^0  \supseteq \hat\YYY^0_j:=\{e_{ji} : i=1,\ldots,n_j-1\} \cup \{-e_{ji} : i=1,\ldots,n_j-1\}.
\end{equation}
In this way, $x_j=0$ is an extreme point of $\CCC_j(\YYY_j^0)$, as it is for $\KKK_j$, and the challenging aspect of the cone is already captured in the first subproblem.
By choosing the coordinate vectors $e_{ji}$ we have
$
\na r_j(e_{ji})^Tx_j = x_{ji} - x_{j0},
$
and the hyperplane constraint \eqref{eq:hyperplane} becomes a very sparse linear constraint.  

When $H^k=0$ in each iteration, this procedure becomes the standard cutting plane algorithm for the SOCP \eqref{eq:socp}.
%
%
It is well-known that the cutting plane algorithm is convergent in the sense that every limit point of the iterates is an optimal solution of the SOCP \eqref{eq:socp}, 
but the convergence is typically slow.
In the following sections, we describe how Algorithm~\ref{alg:basic_primal} is augmented to achieve fast local convergence.  The full method is stated formally in Algorithm~\ref{alg:full}.

\subsection{Identification of extremal-active cones}\label{sec:identification}

We now describe a strategy that enables our algorithm to identify those cones that are extreme-active at a non-degenerate solution $x^*$ within a finite number of iterations, i.e., $\EEE(x^k)=\EEE(x^*)$ for all large $k$.
This will make it possible to apply a second-order method and achieve quadratic local convergence.

Consider the optimality conditions for the QP subproblem~\eqref{eq:qp_Y}:
\begin{subequations}
\begin{align}
c + H^kd^k + A^T\hat \lam^k + \sum_{j\in\DDD(x^k)}\hat\mu_j^k\na_x r_j(x^k) - \hat \nu^k & = 0, \label{eq:opt_pqp_1} \\
A(x^k+d^k)-b \leq 0  &\perp \hat\lam^k\geq 0 \label{eq:opt_pqp_2}, \\
r_j(x_j^k) + \na r_j(x_j^k)^T d_j^k \leq 0 &\perp \hat\mu_j\geq 0, \hspace{1.2cm} j\in\DDD(x^k),\label{eq:opt_pqp_3}\\
\CCC_j(\YYY_j^k)\ni x^k_j+d_j^k & \perp \hat \nu_j^k\in \CCC^\circ_j(\YYY_j^k),\quad j\in\JJJ\label{eq:opt_pqp_4} .
\end{align}
\end{subequations}
%
%
Here, $\hat\lam^k$, $\hat\mu^k_j$, and $\hat \nu_j^k$ are the multipliers corresponding to the constraints in \eqref{eq:qp_Y}; for completeness, we define $\hat\mu_j^k=0$ for $j\in\JJJ\setminus\DDD(x^k)$.
In \eqref{eq:opt_pqp_1}, $\na_x r_j(x^k)$ is the vector in $\RR^n$ that contains $\na r_j(x_j^k)$ in the elements corresponding to $x_j$ and is zero otherwise.  Similarly, $\hat\nu^k\in\RR^n$ is equal to $\hat\nu^k_j$ in the elements corresponding to $x_j$ for all $j\in\JJJ$ and zero otherwise.

Let us define
\begin{equation}\label{eq:hatYregular}
\hat\YYY_j^k := 
\begin{cases}
\YYY_j^k\cup\{x^k_j\}, & \text{ if } j\in\DDD(x^k), \\ 
\YYY_j^k, & \text{ if } j\in\JJJ\setminus\DDD(x^k).
\end{cases}
\end{equation}
It is easy to verify that, for $j\in\DDD(x^k)$, $\na r_j(x_j^k)x_j^k=r_j(x_j^k)$ and hence
$
r_j(x_j^k)^T(x_j^k+d^k) \leq 0
$
from \eqref{eq:opt_pqp_3}.  As a consequence we obtain $x^k_j+d_j^k\in\CCC_j(\hat\YYY_j^k)$ for all $j\in\JJJ$.  Furthermore, $\hat\nu_j^k\in\CCC^\circ_j(\YYY_j^k)$ implies that 
\[
\hat\nu_j^k = - \sum_{l=1}^{m}\sigma^k_{j,l}\na r_j(y^k_{j,l}) + \eta^k_j e_{j0} 
\]
for suitable values of $\sigma^k_{j,l}\geq0$ and $\eta^k_j\geq0$.  
Then $\hat z_j^k:= -\hat \mu_j^k \nabla r_j(x^k) + \hat \nu_j^k\in\CCC^\circ_j(\hat\YYY_j^k)$ and
\begin{equation}\label{eq:hatz}
\hat z^k=c + H^kd^k + A^T\hat \lam^k
\end{equation}
from \eqref{eq:opt_pqp_1}.
In conclusion, if $(d,\hat\lam^k, \hat\mu^k, \hat \nu^k)$ is a primal-dual solution of the QP subproblem \eqref{eq:qp_Y}, then $(d,\hat\lam^k, \hat z^k)$ satisfies the conditions
\begin{subequations}\label{eq:opt_qp}
\begin{align}
c + H^kd^k + A^T\hat \lam^k - \hat z^k & = 0, \label{eq:opt_qp_1} \\
A(x^k+d^k)-b \leq 0  &\perp \hat\lam^k\geq 0  \label{eq:opt_qp_2}, \\
\CCC_j(\hat\YYY_j^k)\ni x^k_j+d_j^k & \perp \hat z_j^k\in \CCC^\circ_j(\hat\YYY_j^k), \quad j\in\JJJ, \label{eq:opt_qp_4} 
\end{align}
\end{subequations}
which more closely resembles the SOCP optimality conditions \eqref{eq:opt_socp}.  Our algorithm maintains primal-dual iterates $(x^{k+1},\hat\lam^k,\hat z^k)$ that are updated based on \eqref{eq:opt_qp}.



Suppose that strict-complementarity holds at a primal-dual solution $(x^*,\lam^*,z^*)$ of the SOCP \eqref{eq:socp} and that $(x^{k+1},\hat\lam^k,\hat z^k)\to(x^*,\lam^*,z^*)$.
If $j\not\in\EEE(x^*)$ then $x_j^*\in\KKK_j$ implies $x^*_{j0}>0$.  As $x^k_j$ converges to $x^*_j$, we have $x^k_{j0}>0$ and therefore $j\not\in\EEE(x^k)$ for sufficiently large $k$.  This yields $\EEE(x^k)\subseteq\EEE(x^*)$.
We now derive a modification of Algorithm~\ref{alg:basic_primal} that ensures that $\EEE(x^*)\subseteq\EEE(x^k)$ for all sufficiently large $k$ under Assumption~\ref{ass:socpnondegen}.

Consider any $j\in\EEE(x^*)$.  We would like to have
\begin{equation}\label{eq:Ckhat}
\hat z_j^k\in \interior(\CCC^\circ_j(\hat\YYY_j^k))
\end{equation}
for all large $k$, since then complementarity in \eqref{eq:opt_qp_4} implies that $x_j^{k+1}=x_j^k+d^k_j=0$ and hence $j\in\EEE(x^{k+1})$ for all large $k$.
We will later show that Assumption~\ref{ass:socpnondegen} implies that $\hat z_j^k \to z^*_j$ and that there exists a neighborhood $N_\epsilon(z^*_j)
= \{z_j\in\RR^{n_j}: \|z_j-z^*_j\|\leq\epsilon\}
$ 
of $z^*_j$ so that $z_j\in \interior(\CCC^\circ_j(\hat\YYY_j^0\cup\{-y_j\}))$ if $z_j,y_j\in N_\epsilon(z^*_j)$; see Remark~\ref{rem:zstartinterior}.
This suggests that some vector close to $-z_j^*$ should eventually be included in $\hat\YYY_j^k$ because then \eqref{eq:Ckhat} holds when $\hat z_j^k$ is close enough to $z^*_j$.
For this purpose, the algorithm computes 
\begin{equation*}
\check z^k = c + A^T\hat\lam^k, 
\end{equation*}
which also converges to $z^*_j$ (see \eqref{eq:opt_socp_1}), and 
sets $\YYY^{k+1}_j$ to $\YYY^+_{du,j}(\YYY^k_j,x^k_j,\check z^k_j)$, where
\begin{equation}\label{eq:YYYplus_du}
\YYY^+_{du,j}(\YYY_j,x_j,z_j) =
 \begin{cases}
  \YYY_j\cup \{-z_j\} & \text{ if } x_j\neq 0, \bar z_j\neq0 \text{ and } r_j(z_j)<0, \\
  \YYY_j & \text{ otherwise.}
 \end{cases}
\end{equation}
The update is skipped when $x^k_j=0$ (because then $j$ is already in $\EEE(x^k)$ and no additional hyperplane is needed), and  when $\bar{\check z}^k_j=0$ or $r_j(\check z^k_j)\geq0$, which might indicate that $z^*_j\not\in\interior(\KKK_j)$ and $j\not\in\EEE(x^*)$.

\subsection{Fast NLP-SQP steps}\label{sec:fastSQPsteps}

Now that we have a mechanism in place that makes sure that the extremal-active cones are identified in a finite number of iterations%
, 
we present a strategy that emulates the basic SQP Algorithm~\ref{alg:SQP} and automatically takes quadratically convergent SQP steps, i.e., solutions of the SQP subproblem \eqref{eq:sqp_subproblem}, close to $x^*$.
For the discussion in this section, we again assume that $x^*$ is a unique solution at which Assumption~\ref{ass:socpnondegen} holds.

Suppose that $\EEE(x^k)=\EEE(x^*)$ for large $k$ due to the strategy discussed in Section~\ref{sec:identification}.  This means that the outer approximation \eqref{eq:qp_Y_5} of $\KKK_j$ for $j\in\EEE(x^*)$ is sufficient to fix $x_j^{k}$ to zero and is therefore equivalent to the constraint \eqref{eq:sqp_subproblem_3} in the basic SQP subproblem.  However, \eqref{eq:qp_Y} includes the outer approximations for all cones, including those for $j\not\in\EEE(x^*)$, which are not present in \eqref{eq:sqp_subproblem}.  Consequently, the desired SQP step from \eqref{eq:sqp_subproblem} might not be feasible for \eqref{eq:qp_Y}.

As a remedy, at the beginning of an iteration, the algorithm first computes an NLP-SQP step as an optimal solution $d^{S,k}$ of a relaxation of \eqref{eq:qp_Y}, 
\begin{subequations}\label{eq:newton_qp}
\begin{align}
\min_{d\in\RR^n} \ & c^Td + \tfrac12 d^TH^k d\\
\text{s.t.} \ & A(x^k+d) \leq b\\
& r_j(x_j^k) + \na  r_j(x_j^k)^Td_j  \leq 0,   \qquad j\in\DDD(x^k)\label{eq:newton_qp3}\\
& x_{j0}^k+d_{j0}  \geq 0,   \hspace{2.5cm} j\in\DDD(x^k)\setminus\hat\EEE^k\label{eq:newton_qp4}\\
& x_j^k + d_j \in \CCC_j(\YYY^k_j)\hspace{2.1cm}j\in\hat\EEE^k,
\end{align}
\end{subequations}
where $\hat\EEE^k=\EEE(x^k)$.
In this way, the outer approximations are imposed only for the currently extremal-active cones, while for all other cones only the linearization \eqref{eq:newton_qp3} is considered, just like in \eqref{eq:sqp_subproblem}, with the additional restriction \eqref{eq:newton_qp4} that ensure $x^{k+1}_{j0}\geq0$.
Let $\hat\lam^k$, $\hat\mu_j^k$, $\hat\eta^k_j$, and $\hat\nu^k_j$ be the optimal  corresponding to the constraints in \eqref{eq:newton_qp} (set to zero for non-existing constraints) and define $\hat z^k$ as in \eqref{eq:hatz}. Then the optimality conditions \eqref{eq:opt_qp} hold again, this time with $d^k=d^{S,k}$, but instead of \eqref{eq:hatYregular} we have
\begin{equation}\label{eq:hatYSQP}
\hat\YYY_j^k := 
\begin{cases}
\{x^k_j\} & \text{ if } j\in\DDD(x^k)\setminus\hat\EEE^k, \\ 
\YYY_j^k\cup\{x^k_j\}& \text{ if } j\in\hat\EEE^k\cap\DDD(x^k), \\ 
\YYY_j^k & \text{ if } j\in\hat\EEE^k\setminus\DDD(x^k).
\end{cases}
\end{equation}



When $x^k$ is not close to $x^*$ and $\EEE(x^*)\neq\EEE(x^k)$, QP \eqref{eq:newton_qp} might result in poor steps that go far outside of $\KKK_j$ for some $j\in\DDD(x^k)\setminus\hat\EEE^k$ and undermine convergence.
Therefore, we iteratively add more cones to $\hat\EEE^k$ until 
\begin{equation}\label{eq:hatEEEcondition}
x_{j0}^k+d_{j0}^{S,k}>0 \text{ only for } j\in\JJJ\setminus\hat\EEE^k, 
\end{equation}
i.e., when a cone is approximated only by its linearization \eqref{eq:newton_qp3}, the step does not appear to target its extreme point. This property is necessary to show that $\EEE(x^k)=\EEE(x^*)$ for all large $k$ also for the case that new iterates are computed from \eqref{eq:newton_qp} instead of \eqref{eq:qp_Y}. 
Note that in the extreme case $\hat\EEE^k=\JJJ$ and \eqref{eq:newton_qp} is identical to \eqref{eq:qp_Y}.
This loop can be found in Steps \ref{s:whileEEE}--\ref{e:whileEEE} in Algorithm~\ref{alg:full}.

Since there is no guarantee that  \eqref{eq:newton_qp} yields iterates that converge to $x^*$, the algorithm discards the NLP-SQP step in certain situations and falls back to the original method to recompute a new step from \eqref{eq:qp_Y}, as in the original method.  In Section~\ref{s:penalty} we describe how we use the exact penalty function \eqref{eq:merit} to determine when this is necessary.

\subsection{Hessian matrix}

Motivated by \eqref{eq:sqphessian}, we compute the Hessian matrix $H^k$ in \eqref{eq:qp_Y} and \eqref{eq:newton_qp} from
\begin{equation}\label{eq:Hessian}
    H^k = \sum_{j\in\DDD(x^k)}\mu^k_j\na^2_{xx}r_j(x^k),
\end{equation}
where $\mu_j^k\geq 0$ are multiplier estimates for the nonlinear constraint \eqref{eq:nlp_r_3}.
Because $\na^2r_j(x_j^k)$ is positive semi-definite and $\mu_j^k\geq0$, also $H^k$ is positive semi-definite.

In the final phase, when we intend to emulate the basic SQP Algorithm~\ref{alg:SQP}.  Therefore, we set $\mu_j^{k+1}=\hat\mu^k_j$ for $j\in\DDD(x^k)$, where $\hat\mu_j^k$ are the optimal multipliers for \eqref{eq:newton_qp3}, when the fast NLP-SQP step was accepted.
But we also need to define a value for $\mu_j^{k+1}$ when the step is computed from \eqref{eq:qp_Y} where, in addition to the linearization of $r_j$, hyperplanes \eqref{eq:qp_Y_5} are used to approximate all cones.
By comparing the optimality conditions of the QPs \eqref{eq:qp_Y} and \eqref{eq:sqp_subproblem}  we now derive an update for $\mu_j^{k+1}$.

Suppose that $j\in\DDD(x^{k+1})\cap\DDD(x^k)$. Then \eqref{eq:opt_pqp_1} yields
\begin{equation}\label{eq:hkmudis}
c_j + H_{jj}^kd^k_j + A_j^T\hat \lam^k + \hat\mu_j^k\na r_j(x_j^k) - \hat \nu^k_j= 0,
\end{equation}
where $H^k_{jj} = \mu_j^k\na^2r_j(x^k_j)$ because of \eqref{eq:Hessian}.
Here, the dual information for the nonlinear constraint is split into $\hat\mu_j^k$ and $\hat \nu^k_j$ and needs to be condensed into a single number, $\mu_j^{k+1}$, so that we can compute $H^k$ from \eqref{eq:Hessian} in the next iteration.

Recall that, in the basic SQP Algorithm~\ref{alg:SQP}, the new multipliers $\mu_j^{k+1}$ are set to the optimal multipliers of the QP \eqref{eq:sqp_subproblem}, which satisfy
\begin{equation}\label{eq:hkmudis2}
c_j + H_{jj}^kd_j^k + A_j^T\hat\lam^k + \mu_j^{k+1}\na r_j(x_j^k) = 0.
\end{equation}
A comparison with \eqref{eq:hkmudis} suggests to choose $\mu_j^{k+1}$ so that
$
 \mu_j^{k+1}\na r_j(x_j^k) \approx \hat\mu_j^k\na r_j(x_j^k) - \hat \nu^k_j.
$
Multiplying both sides with $\na r_j(x^k_j)^T$ and solving for $\mu^{k+1}$ yields
\begin{equation*}
\mu^{k+1}_j = \hat \mu_j^k - \frac{\na r_j(x_j^k)^T\hat \nu^k_j}{\|\na r_j(x_j^k)\|^2}.
\end{equation*}
Note that $ \mu_j^{k+1} = \hat \mu_j^k$ if the outer approximation constraint \eqref{eq:qp_Y_5} is not active and therefore $\hat\nu_j^k=0$ for $j$. In this case, we recover the basic SQP update, as desired.

Now suppose that $j\in\DDD(x^{k+1})\setminus\DDD(x^k)$. Again comparing \eqref{eq:hkmudis} with \eqref{eq:hkmudis2} suggests a choice so that
$
 \mu_j^{k+1}\na r_j(x_j^{k+1}) \approx - \hat \nu^k_j,
$
where we substituted $\na r_j(x^k_j)$ by $\na r_j(x^{k+1}_j)$ because the former is not defined for $j\not\in\DDD(x^k)$.
In this case, multiplying both sides with $\na r_j(x^{k+1}_j)^T$ and solving for $\mu^{k+1}$ yields
\begin{equation*}
\mu^{k+1}_j = - \frac{\na r_j(x_j^{k+1})^T\hat \nu^k_j}{\|\na r_j(x_j^{k+1})\|^2}.
\end{equation*}
In summary, in each iteration in which \eqref{eq:qp_Y} determines the new iterate, we update 
\begin{equation}\label{eq:mu_update}
\mu^{k+1}_j = \begin{cases}
 \hat \mu_j^k - \frac{\na r_j(x_j^k)^T\hat \nu^k_j}{\|\na r_j(x_j^k)\|^2} & j\in\DDD(x^{k+1})\cap\DDD(x^k)\\[3ex]
 - \frac{\na r_j(x_j^{k+1})^T\hat \nu^k_j}{\|\na r_j(x_j^{k+1})\|^2} & j\in\DDD(x^{k+1})\setminus\DDD(x^k) \\
 0 &\text{otherwise.}
\end{cases}
\end{equation}

The choice above leads to quadratic convergence for non-degenerate instances, but it is common for the global convergence analysis of SQP methods to permit any positive semi-definite Hessian matrix $H^k$, as long as it is bounded.  Since we were not able to exclude the case that $\mu_j^k$ or $1/x_{j0}^k$ are unbounded for some cone $j\in\JJJ$, in which case $H^k$ defined in \eqref{eq:Hessian} is unbounded, we fix a large threshold $c_H>0$ and rescale the Hessian matrix according to
\begin{equation}\label{eq:Hrescale}
H_k \gets H_k \cdot \min\{1, c_H/\|H_k\|\}
\end{equation}
so that $\|H_k\|\leq c_H$ in every iteration.
In this way, global convergence is guaranteed, but the fast local convergence rate might be impaired if $c_H$ is chosen too small so that $H^k$ defined in \eqref{eq:Hessian} must be rescaled.  Therefore, in practice, we set $c_H$ to a very large number and \eqref{eq:Hrescale} is never actually triggered in our numerical experiments.

\subsection{Penalty function}\label{s:penalty}

The steps computed from \eqref{eq:qp_Y} and \eqref{eq:newton_qp} do not necessarily yield a convergent algorithm and a safeguard is required to force the iterates into a neighborhood of an optimal solution.
Here, we utilized the exact penalty function \eqref{eq:merit} and accept a new iterate only if the sufficient decrease condition \eqref{eq:armijo} holds.

As discussed in Section~\ref{sec:fastSQPsteps}, at the beginning of an iteration, the algorithm first computes an NLP-SQP step $d^{S,k}$ from \eqref{eq:newton_qp}.  The penalty function can now help us to decide whether this step makes sufficient progress towards an optimal solution, and we only accept the trial point $\hat x^{k+1}=x^k+d^{S,k}$ as a new iterate if \eqref{eq:armijo} holds with $d=d^{S,k}$.

If the penalty function does not accept $d^{S,k}$, there is still a chance that $d^{S,k}$ is making rapid progress towards the solution, but, as discussed in Section~\ref{sec:localSQP}, the Maratos effect is preventing the acceptance of $d^{S,k}$. 
As a remedy, we compute, analogously to \eqref{eq:sqp_soc}, a second-order correction step $s^k$ for \eqref{eq:newton_qp} as a solution of
\begin{equation}\label{eq:newton_soc}
\begin{aligned}
\min_{s\in\RR^n} \ & c^T(d^{S,k}+s) + \tfrac12 (d^{S,k}+s)^TH^k (d^{S,k}+s)\\
\text{s.t.} \ & A(x^k+d^{S,k}+s) \leq b,\\
& r_j(x_j^k+d_j^{S,k}) + \na  r_j(x_j^k+d_j^{S,k})^Ts_j  \leq 0,   \qquad j\in\DDD(x^k),\\
& x_{j0}^k + d_{j0}^{S,k} + s_{j0} \geq 0,  \hspace{3.55cm} j\in\DDD(x^k)\setminus\hat\EEE^k,\\
& x_j^k + d^{S,k}_j + s_j \in \CCC_j(\YYY^k_j),\hspace{3.0cm} j\in\hat\EEE^k,
\end{aligned}
\end{equation}
and accept the trial point $\hat x^{k+1} = x^k + d^{S,k}+s^k$ if it satisfies \eqref{eq:armijo} with $d=d^{S,k}$.
Let again $\hat\lam^k$, $\hat\mu_j^k$, $\hat\eta^k_j$, and $\hat\nu^k_j$ denote the optimal multipliers in \eqref{eq:newton_soc} and define $\hat z^k$ as in \eqref{eq:hatz}.
The optimality conditions \eqref{eq:opt_qp} still hold, this time with $d^k=d^{S,k}+s^k$ and
\begin{equation}\label{eq:hatYSOC}
\hat\YYY_j^k := 
\begin{cases}
\{x^k_j+d_j^{S,k}\}, & \text{ if } j\in\DDD(x^k)\setminus\hat\EEE^k, \\ 
\YYY_j^k\cup \{x^k_j+d_j^{S,k}\}, & \text{ if } j\in\DDD(x^k)\cap \hat\EEE^k, \\ 
\YYY_j^k, & \text{ if } j\in\hat\EEE^k.
\end{cases}
\end{equation}
%

If neither $d^{S,k}$ nor $d^{S,k}+s^k$ has been accepted, we give up on fast NLP-SQP steps and instead revert to QP \eqref{eq:qp_Y} which safely approximates every cone with an outer approximation.
However, the trial point $\hat x^{k+1}=x^k+d^k$ with the step $d^k$ obtained from \eqref{eq:qp_Y} does not necessarily satisfy \eqref{eq:armijo}.  In that case, the algorithm adds $x^k+d^k$ to $\YYY_j^k$ to cut off $x^k+d^k$ and resolves \eqref{eq:qp_Y} to get a new trial step $d^k$.  In an inner loop (Steps~\ref{s:begininner}--\ref{s:endinner}), this procedure is repeated until, eventually, a trial step is obtained that satisfies \eqref{eq:armijo}.
We will show that \eqref{eq:armijo} holds after a finite number of iterations of the inner loop.

It remains to discuss the update of the penalty parameter estimate $\rho^k$. 
One can show (see Lemma~\ref{lem:modeldiff}) that an optimal solution of $x^*$ of the SOCP with multipliers $z^*$ is a minimizer of $\phi(\;\cdot,\rho)$ over the set $X$ defined in \eqref{eq:setX} if $\rho>\|z^*_{\JJJ,0}\|_{\infty}$, where $z^*_{\JJJ,0}=(z^*_{1,0},\ldots,z^*_{p,0})^T$.
Since $z^*$ is not known \textit{a priori}, the algorithm uses the update rule $\rho^k = \rho_{\text{new}}(\rho^{k-1},z^k)$ where
\begin{equation}\label{eq:rhoupdate}
	\rho_{\text{new}}(\rho_{\text{old}},z):=\begin{cases}
	\rho_{\text{old}} & \text{ if } \rho_{\text{old}} > \|z_{\JJJ,0}\|_\infty \\
	c_{\text{inc}}\cdot\|z_{\JJJ,0}\|_\infty & \text{ otherwise,}
	\end{cases}
\end{equation}
with $c_{\text{inc}}>1$. We will prove in Lemma~\ref{lem:zbounded} that the sequence $\{z^k\}_{k=1}^\infty$ is bounded under Slater's constraint qualification.  Therefore, this rule will eventually settle at a final penalty parameter $\rho^\infty$ that is not changed after a finite number of iterations.

During an iteration of the algorithm, several trial steps may be considered and a preliminary parameter value is computed from \eqref{eq:rhoupdate} for each one.  At the end of the iteration, the parameter value corresponding to the accepted trial step is stored.
Note that the acceptance test for the second-order correction step from \eqref{eq:newton_soc} needs to be done with the penalty parameter computed for the regular NLP-SQP step from \eqref{eq:newton_qp}.

\subsection{Complete algorithm}
\label{s:complete}

\begin{algorithm}[!h]
    \caption{SQP Algorithm for SOCP.}\label{alg:full}
    \begin{algorithmic}[1]
        \Require Initial iterate $x^0\in X$ with $x_{j,0}\geq0$, multipliers $\mu_j^0\in\RR_+$, penalty parameter $\rho^{-1}>0$;  constants $c_{\text{dec}}\in (0,1)$, $c_{\text{inc}}>1$, and $c_H>0$.
        \State Initialize $\YYY_j^0$ so that \eqref{eq:Y0} is satisfied.
        \For{$k=0,1,2,\ldots$}
        \State Compute $H^k$ using \eqref{eq:Hessian}. Rescale according to \eqref{eq:Hrescale} if $\|H^k\|>c_H$.\label{s:calcH}
        \State Set $\hat\EEE^k\gets\EEE(x^k)$.
        \State Compute $d^{S,k}$, $\hat\lam^k$, $\hat\mu^{k}$, $\hat z^{k}$ from \eqref{eq:newton_qp} and \eqref{eq:hatz} and set $\hat x^{k+1} = x^k + d^{S,k}$. \label{s:newton}
        \While{$\{j\in \JJJ: x_{j0}^k+d^{S, k}_{j0} = 0\}\not\subseteq\hat \EEE^k $}\label{s:whileEEE}
        \State Set $\hat \EEE^k \gets \hat \EEE^k \cup \{j\in \JJJ: x_{j0}^k+d^{S, k}_{j0} = 0\}$. \label{s:whileEEE1}
        \State Recompute $d^{S,k}$, $\hat\lam^k$, $\hat\mu^{k}$, $\hat z^{k}$ from \eqref{eq:newton_qp} and \eqref{eq:hatz} and set $\hat x^{k+1} = x^k + d^{S,k}$. 
        \EndWhile \label{e:whileEEE}
        \State Compute candidate penalty parameter $\rho^{k}=\rho_\text{new}(\rho^{k-1},\hat z^k)$, see \eqref{eq:rhoupdate}.
        \If{\eqref{eq:armijo} holds for $d=d^{S,k}$}\label{s:testNewton}
        \State Set $\YYY^{k+1}_j\gets \YYY^+_{pr,j}(\YYY^{k}_j,x_j^{k})$ using \eqref{eq:YYYplus_pr} and set $d^k=d^{S,k}$.
        \State Set $\mu^{k+1}=\hat \mu^k$ and go to Step~\ref{s:updateYdual}.
        \EndIf
        \State Compute $s^k$, $\hat\lam^k$, $\hat\mu^{k}$, $\hat z^{k}$ from \eqref{eq:newton_soc} and \eqref{eq:hatz} and set $\hat x^{k+1} = x^k + d^{S,k} +s^k$. \label{s:soc} 
        \If{\eqref{eq:armijo} holds for $d=d^{S,k}$ and $\{j\in \JJJ: x_{j0}^k+d^{S, k}_{j0} + s^k= 0\}\subseteq\hat \EEE^k $}  \label{s:testSOC}
        \State Set $\YYY^{k+1}_j\gets\YYY^+_{pr,j}(\YYY^{k}_j,x_j^{k})$  and $d^k=d^{S,k}$.
        \State Set $\mu^{k+1}=\hat \mu^k$ and go to Step~\ref{s:updateYdual}.
        \EndIf \label{s:exit_soc}
        \State Set $\YYY_j^{k,0}\gets\YYY^k_j$.
        \For{$l=0,1,2,\ldots$} \label{s:begininner}
        \State Compute $d^{k,l}$, $\hat\lam^k$, $\hat\mu^k$, $\hat z^k$ from \eqref{eq:qp_Y} and \eqref{eq:hatz} and set $\hat x^{k+1} = x^k + d^{k,l}$.\label{s:regularstep}
        \State Compute candidate penalty parameter $\rho^{k}=\rho_\text{new}(\rho^{k-1},\hat z^k)$.
        \If{\eqref{eq:armijo} holds for $d=d^{k,l}$}\label{s:testinner}
        \State Set $\YYY_j^{k+1}\gets \YYY^+_{pr,j}(\YYY^{k,l}_j,x_j^{k})$ and $d^{k} = d^{k,l}$.\label{s:defdR}
        \State Go to Step~\ref{s:updatemu}.
        \EndIf
        \State Set $\YYY_j^{k+1}\gets \YYY^+_{pr,j}(\YYY^{k,l}_j,\hat x_j^{k+1})$, see \eqref{eq:YYYplus_pr}.\label{s:updateYpr}
        \State Compute $\check z^{k} = c + A^T\hat \lam^k$.
        \State Update $\YYY_j^{k,l+1}\gets \YYY^+_{du,j}(\YYY^{k,l+1}_j,\hat x_j^{k+1},\check z^k_j)$, see \eqref{eq:YYYplus_du}.\label{s:updateYinner}
        \EndFor\label{s:endinner}
        \State Compute $\mu^{k+1}$ from \eqref{eq:mu_update}.\label{s:updatemu}
        \State Compute $\check z^{k} = c + A^T\hat \lam^k$ and update $\YYY_j^{k+1}\gets \YYY^+_{du,j}(\YYY^{k+1}_j,x_j^k,\check z^k_j)$.\label{s:updateYdual}
        \State Set $x^{k+1}\gets \hat x^{k+1}$.\label{s:updatex}
        \State If $(x^{k+1},\hat\lam^{k},\hat z^{k})$ satisfy \eqref{eq:opt_socp}, \textbf{stop}.\label{s:stop}
        \EndFor
    \end{algorithmic}
\end{algorithm}

The complete method is stated in Algorithm~\ref{alg:full}.  To keep the notation concise, we omit ``for all $j\in\JJJ$'' whenever the index $j$ is used. We assume that all QPs in the algorithm are solved exactly.

Each iteration begins with the computation of the fast NLP-SQP step where an inner loop repeatedly adds cones to $\hat\EEE^k$ until \eqref{eq:hatEEEcondition} holds.
If the step achieves a sufficient decrease in the penalty function, the trial point is accepted.
Otherwise, the second-order correction for the NLP-SQP step is computed and accepted if it yields a sufficient decrease for the NLP-SQP step.
Note that the second-order correction step is discarded if it does not satisfy \eqref{eq:hatEEEcondition} since otherwise finite identification of $\EEE(x^*)$ cannot be guaranteed.
If none of the NLP-SQP steps was acceptable, the algorithm proceeds with an inner loop in which hyperplanes cutting off the current trial point are repeatedly added until the penalty function is sufficiently decreased.
No matter which step is taken, both $x_j^k$ and $\check z_j^k$ are added to $\YYY_j^k$ according to the update rules \eqref{eq:YYYplus_pr} and \eqref{eq:YYYplus_du} and the multiplier $\mu^k$ for the nonlinear constraints is updated.
%


In most cases, a new QP is obtained by adding only a few constraints to the most recently solved QP, and a hot-started QP solver will typically compute the new solution quickly.
For example, in each inner iteration in Steps~\ref{s:whileEEE}--\ref{e:whileEEE}, hyperplanes for the polyhedral outer approximation for cones augmenting $\hat\EEE^k$ are added to QP \eqref{eq:newton_qp}.
Similarly, each inner iteration in Steps~\ref{s:begininner}--\ref{s:endinner} adds one cutting plane for a violated cone constraint.
In Steps~\ref{s:newton} and \ref{s:soc}, some constraints are removed compared to the most recently solved QP, but also this structure could be utilized.

The algorithm might terminate because one of QPs solved for the step computation is infeasible. Since the feasible regions of the QP are outer approximations of the SOCP~\eqref{eq:socp}, this proves that the SOCP instance is infeasible; see also Remark~\ref{rem:infeasible}.

\section{Convergence analysis}\label{sec:theory}

\subsection{Global convergence}

In this section, we prove that, under a standard regularity assumption, all limit points of the sequence of iterates are optimal solutions of the SOCP, if the algorithm does not terminate with an optimal solution in Step~\ref{s:stop}. We also explore what happens when the SOCP is infeasible.

We make the following assumption throughout this section.
\begin{assumption}
\label{as:socp_bounded} \label{ass:boundedness}
The set $X$ defined in \eqref{eq:setX} is bounded.
\end{assumption}

Since $x^0\in X$ by the initialization of Algorithm~\ref{alg:full} and any step satisfies \eqref{eq:opt_pqp_2}, we have $x^k\in X$ for all $k$.
Similarly, \eqref{eq:opt_qp_4} and \eqref{eq:CCC_j} imply that
\begin{equation}\label{eq:x0feasible}
x^{k}_{j0} \geq 0 \text{ for all } k\geq0 \text{ and } j\in\JJJ.
\end{equation}
%

We start the analysis with some technical results that quantify the decrease in the penalty function model.
\begin{lemma}\label{lem:modeldiff}
Consider an iteration $k$ and let $d^k$ be computed in Step~\ref{s:newton} or Step~\ref{s:regularstep} in Algorithm~\ref{alg:full}.
Further let $\rho^k>\rho_{\min}^k$, where $\rho_{\min}^k=\|\hat z_{\JJJ,0}^k\|_\infty$ with $\hat z^k$ defined in \eqref{eq:hatz}.  Then the following statements are true.
\begin{enumerate}[(i)]
\item We have
\begin{equation*}
m^k(x^{k}+d^k;\rho^k) - m^k(x^k;\rho^k)\leq -(d^k)^TH^kd^k -(\rho^k-\rho_{\min}^k)\sum_{j\in\JJJ}[r_j(x^k_j)]^+\leq 0.
\end{equation*}\label{en:model_decrease}
\item If $x^k$ is not an optimal solution of the SOCP, then
\begin{equation}\label{eq:modeldiffnegative}
m^k(x^{k}+d^k;\rho^k) - m^k(x^k;\rho^k) < 0.
\end{equation}\label{en:modeldiffnegative}
\end{enumerate}
\end{lemma}

\begin{proof}
Proof of (\ref{en:model_decrease}): 
Consider any $j\in\DDD(x^k)$. 
Because $d^k$ is a solution of \eqref{eq:qp_Y} or \eqref{eq:newton_qp}, there exist $\hat\lam^k$ and $\hat z^k$ so that the optimality conditions \eqref{eq:opt_qp} hold.
Let $j\in\JJJ$.  Since $\hat z_j^k\in\CCC^\circ(\hat\YYY_j^k)$, the definition \eqref{eq:Cdual} implies that $$\hat z^k_j=- \sum_{l=1}^{m_j^k}\hat\sigma^k_{l,j}\na r_j(y_{j,l}^{k}) + \hat\eta^k_j e_{j0},$$
where $
\hat\YYY^k_j=\left\{y^{k}_{j,1},\ldots,y^{k}_{j,m_j^k}\right\}
$
and $\hat\sigma^k_{l,j}, \hat\eta^k_j\in\RR_+$.

Using \eqref{eq:na_r} we have
$
\hat z_{j0}^k = \sum_{l=1}^{m_j^k}\hat\sigma^k_{l,j} + \hat\eta^k_j\geq  \sum_{l=1}^{m_j^k}\hat\sigma^k_{l,j}.
$
Together with $(x_j^k+d^k_j)^T\hat z_j^k=0$ from \eqref{eq:opt_qp_4} and $(\bar x^k_j)^T\bar y^k_{l,j}\leq\|\bar x^k_j\|\cdot\|\bar y^k_{l,j}\|$ this overall yields
\begin{align*}
-(d_j^k)^T\hat z_j^k &= (x_j^k)^T\hat z_j^k  = x^k_{j0} \hat z_{j0}^k
-\sum_{l=1}^{m_j^k}\hat\sigma^k_{l,j}(\bar x^k_j)^T\frac{\bar y^k_{l,j}}{\|\bar y^k_{l,j}\|}
\geq x^k_{j0} \hat z_{j0}^k - \sum_{l=1}^{m_j^k}\hat\sigma^k_{l,j}\|\bar x_j^k\|\nonumber\\
&\geq x^k_{j0} \hat z_{j0}^k - \hat z^k_{j0}\|\bar x_j^k\| = -z^k_{j0}r_j(x_j^k) \geq -z^k_{j0}[r_j(x_j^k)]^+.\nonumber
\end{align*}
Further, we have from \eqref{eq:opt_qp_2} that
$
0 = (Ax^k+Ad^k-b)^T\hat\lam^k
$ 
and therefore $(d^k)^TA^T\hat\lam^k = -(Ax^k-b)^T\hat\lam^k\geq0$ since $\hat\lam^k\geq0$ and $x^k\in X$.

Using these inequalities and \eqref{eq:opt_qp_1}, the choice of $\rho^k_{\min}$ yields
\begin{align*}
0&=(d^k)^T\left(c + H^kd^k + A^T\hat \lam^k - \hat z^k\right) \\
&\geq c^Td^k + (d^k)^TH^kd^k - \sum_{j\in\JJJ}\hat z^k_{j0}[r_j(x_j^k)]^+\\
&\geq c^Td^k + (d^k)^TH^kd^k - \rho^k_{\min}\sum_{j\in\JJJ}[r_j(x_j^k)]^+.
\end{align*}
Finally, combining this with \eqref{eq:merit_model} and \eqref{eq:qp_Y_3} or \eqref{eq:newton_qp3}, respectively, we obtain
\begin{align*}
m^k(x^{k}+d^k;\rho^k) - m^k(x^k;\rho^k) & = c^Td^k - \rho^k\sum_{j\in\DDD(x^k)}[r_j(x_j^k)]^+ \\
 & = c^Td^k - \rho^k\sum_{j\in\JJJ}[r_j(x_j^k)]^+\\
& \leq -(d^k)^TH^kd^k - (\rho^k-\rho^k_{\min})\sum_{j\in\JJJ}[r_j(x_j^k)]^+.
\end{align*}
For the second equality, we used that $r_j(x_j^k) = 0 - x^k_{j0} \leq 0$  for $j\not\in\DDD(x^k)$ by \eqref{eq:x0feasible} and the definition of $\DDD(x^k)$.
Since $H^k$ is positive semi-definite, $\rho^k>\rho^k_{\min}$, and $[r_j(x^k_j)]^+\geq0$, the right-hand side must be non-positive.

%

Proof of (\ref{en:modeldiffnegative}):
Suppose $x^k\in X$ is not an optimal solution for the SOCP.
If $x^k$ is not feasible for the SOCP, $x^k$ must violate a conic constraint and we have $[r_j(x_j^k)]^+>0$ for some $j\in\JJJ$.  Since $H^k$ is positive semidefinite and $\rho^k-\rho^k_{\min}>0$,
part (\ref{en:model_decrease}) yields \eqref{eq:modeldiffnegative}.

It remains to consider the case when $x^k$ is feasible for the SOCP, i.e., $[r_j(x^k_j)]^+=0$ for all $j$.  To derive a contradiction, suppose that \eqref{eq:modeldiffnegative} does not hold.  Then part (\ref{en:model_decrease}) yields
\begin{align*}
0 &= m^k(x^k+d^k;\rho^k) - m^k(x^k;\rho^k)\\
&=-(d^k)^TH^kd^k - (\rho^k-\rho^k_{\min})\sum_{j\in\JJJ}[r_j(x_j^k)]^+ =-(d^k)^TH^kd^k\leq0.
\end{align*}
Because $H^k$ is positive semi-definite, this implies $H^kd^k=0$.
Further, since also 
\begin{equation*}
0 = m^k(x^k+d^k;\rho^k) - m^k(x^k;\rho^k) \stackrel{\eqref{eq:merit_model}}{=} c^Td^k - \rho^k\sum_{j\in\DDD(x^k)}[r_j(x^k_j)]^+ = c^Td^k,
\end{equation*}
the optimal objective value of \eqref{eq:qp_Y} or \eqref{eq:newton_qp}, respectively, is zero.  At the same time, choosing $d^k=0$ is also feasible for \eqref{eq:qp_Y} or \eqref{eq:newton_qp} and yields the same objective value.  Therefore, also $d^k=0$ is an optimal solution of \eqref{eq:qp_Y} or \eqref{eq:newton_qp} and the optimality conditions \eqref{eq:opt_qp} hold for some multipliers.  Because $\CCC^\circ_j(\hat\YYY_k^k)\subseteq\KKK_j$, the same multipliers and $d^k=0$ show that the optimality conditions of the SOCP \eqref{eq:opt_socp} also hold.  So, $x^k$ is an optimal solution for the SOCP, contradicting the assumption.
\end{proof}

The following lemma shows that the algorithm is well-defined and will not stay in an infinite loop in Steps~\ref{s:begininner}--\ref{s:endinner}.
\begin{lemma}\label{lem:modeldecrease}
Consider an iteration $k$ and let $d^k$ be computed in Step~\ref{s:newton} or Step~\ref{s:regularstep} in Algorithm~\ref{alg:full}. Suppose that $x^k$ is not an optimal solution of the SOCP.  Then
\begin{equation}\label{eq:suffdecdlt}
\varphi(x^k+d^{k,l}; \rho^k) -  \varphi(x^k; \rho^k)\leq c_{\text{dec}}\Big(m^k(x^k+d^{k,l};\rho^k) - m^k(x^k;\rho^k)\Big)
\end{equation}
after a finite number of iterations in the inner loop in Steps~\ref{s:begininner}--\ref{s:endinner}.
\end{lemma}
\begin{proof}

Suppose the claim is not true and let $\{d^{k,l}\}_{l=0}^{\infty}$ be the infinite sequence of trial steps generated in the loop in Steps \ref{s:begininner}--\ref{s:endinner} for which the stopping condition in Step~\ref{s:testinner} is never satisfied, and let $d^{k,\infty}$ be a limit point of  $\{d^{k,l}\}_{l=0}^{\infty}$. 
We will first show that
\begin{equation}\label{eq:rpluszero}
[r_j(x_j^k+ d_j^{k,\infty})]^+=0  \text{ for all } j\in\JJJ.
\end{equation}

Let us first consider the case when $\bar x_j^k+\bar d_j^{k,\infty}=0$ for some $j\in\JJJ$. Then $r_j(x_j^k+d_j^{k,\infty}) = \|\bar x_j^k+\bar d_j^{k,\infty}\| - (x_{j0}^k+d_{j0}^{k,\infty}) = - (x_{j0}^k+d_{j0}^{k,\infty})  \leq 0$ and \eqref{eq:rpluszero} holds.

Now consider the case that $\bar x_j^k+\bar d_j^{k,\infty}\neq 0$ for $j\in\JJJ$. Since $d^{k,\infty}$ is a limit point of  $\{d^{k,l}\}_{l=0}^{\infty}$, there exists a subsequence $\{d^{k,l_t}\}_{t=0}^\infty$ that converges to $d^{k,\infty}$.
We may assume without loss of generality that $\bar x_j^k+\bar d_j^{k,l_t}\neq 0$ for all $t$.  Then, for any $t$, by Step~\ref{s:updateYinner}, $x_j^k+ d_j^{k,l_t}\in\YYY_j^{k,l_{t+1}}$.  In the inner iteration $l_{t+1}$, the trial step $d_j^{k,l_{t+1}}$ is computed from \eqref{eq:qp_Y} and satisfies $x_j^k+ d_j^{k,l_{t+1}} \in \CCC_j(\YYY_j^{k,l_t})$, which by definition \eqref{eq:CCC_j} implies
\[
\na  r_j(x_j^k+ d_j^{k,l_t})^T(x_j^k+ d_j^{k,l_{t+1}}) \leq 0.
\]
Taking the limit $t\to\infty$ and using the fact that $\na r_j(v_j)^Tv_j = r_j(v_j)$ for any $v_j\in\KKK_j$ yields
$$
r_j(x_j^k+ d_j^{k,\infty}) = \na  r_j(x_j^k+ d_j^{k,\infty})^T(x_j^k+ d_j^{k,\infty}) \leq 0,
$$
proving \eqref{eq:rpluszero}.  In turn \eqref{eq:rpluszero} implies that the ratio
\begin{align*}
\frac{\varphi(x^k+d^{k,l}; \rho^k) -  \varphi(x^k; \rho^k)}{m^k(x^k+d^{k,l};\rho^k) - m^k(x^k;\rho^k)}
=
\frac{c^Td^{k,l} + \rho^k([r_j(x_j^k+ d_j^{k,l})]^+-[r_j(x_j^k)]^+)}{c^Td^{k,l} - \rho^k[r_j(x_j^k)]^+}
\end{align*}
converges to 1.  Note that the ratio is well-defined since $m^k(x^k+d^{k,l};\rho^k) - m^k(x^k;\rho^k)<0$ due to Lemma~\ref{lem:modeldecrease}(\ref{en:modeldiffnegative}). It then follows that \eqref{eq:suffdecdlt} is true for sufficiently large $l$.
\end{proof}

\begin{lemma}\label{lem:xlimit}
Suppose that there exists $\rho^{\infty}>0$ so that $\rho^k= \rho^{\infty}>0$ for all large $k$.
Then any limit point of $\{x^k\}_{k=0}^\infty$ is an optimal solution of the SOCP \eqref{eq:socp}.
\end{lemma}
\begin{proof}
From \eqref{eq:armijo} and the updates in the algorithm, we have that
\begin{align*}
\varphi(x^{k+1};\rho^\infty) - \varphi(x^{K_\rho};\rho^\infty) &  =\sum_{t=K_\rho}^k\Big(\varphi(x^{t+1};\rho^\infty) - \varphi(x^t;\rho^\infty)\Big) \\
 &  \leq c_{\text{dec}}\sum_{t=K_\rho}^k\Big(m^t(x^{t}+d^t;\rho^\infty) - m^t(x^t;\rho^\infty)\Big)
\end{align*}
for $k\geq K_\rho$.
Since the SOCP cannot be unbounded below by Assumption~\ref{ass:boundedness}, the left-hand side is bounded below as $k\to\infty$.   Lemma~\ref{lem:modeldiff}(\ref{en:model_decrease}) shows that all summands are non-positive and we obtain
\begin{equation}\label{eq:diffm_zero}
\lim_{k\to\infty}\Big(m^k(x^{k}+d^k;\rho^\infty) - m^k(x^k;\rho^\infty)\Big) = 0.
\end{equation}
Using Lemma~\ref{lem:modeldiff}(\ref{en:model_decrease}), we also have
\[
\lim_{k\to\infty} \biggl((d^k)^TH^kd^k + (\rho^\infty-\rho_{\min}^k)\sum_{j\in\JJJ}[r_j(x^k_j)]^+\biggr) = 0.
\]
Since $H^k$ is positive semi-definite and $\rho^\infty-\rho^k_{\min}\geq\rho^\infty-\rho^\infty_{\min}>0 $, this implies that $[r_j(x^k_j)]^+\to0$ for all $j\in\JJJ$, i.e., all limit points of $\{x^k\}_{k=0}^\infty$ are feasible.
This also yields $\lim_{k\to\infty}(d^k)^TH^k d^k=0$, and since $H^k$ is positive semi-definite and uniformly bounded due to \eqref{eq:Hrescale}, we have
\begin{equation}\label{eq:Hdlimit}
\lim_{k\to\infty}H^k d^k=0.
\end{equation}

Using \eqref{eq:diffm_zero} together with \eqref{eq:merit_model} and $[r_j(x^k_j)]^+\to0$, we obtain 
\begin{equation}\label{eq:ctdlim}
0 = \lim_{k\to\infty}\biggl( c^Td^k - \rho^\infty\sum_{j\in\DDD(x^k)}[r_j(x^k_j)]^+\biggr) = \lim_{k\to\infty} c^Td^k.
\end{equation}

Now let $x^*$ be a limit point of $\{x^k\}_{k=0}^\infty$.  Since $X$ is bounded, $d^k$ is bounded, and consequently there exists a subsequence $\{k_t\}_{t=0}^\infty$ of iterates so that $x^{k_t}$ and $d^{k_t}$ converge to $x^*$ and $d^\infty$, respectively, for some limit point $d^\infty$ of $d^k$. 
Define $g^{k_t}= H^{k_t}d^{k_t}$ for all $t$.  Then, looking at the QP optimality conditions \eqref{eq:opt_qp}, we see that $d^{k_t}$ is also an optimal solution of the linear optimization problem
\begin{equation}\label{eq:cplusgLP}
\begin{aligned}
\min_{d\in\RR^n} \ & (c+g^{k_t})^Td \\
\text{s.t.} \ & A(x^{k_t}+d) \leq b, \\
& x_j^{k_t} + d_j \in \CCC_j(\hat\YYY^{k_t}_j), \qquad j\in\JJJ.
\end{aligned}
\end{equation}

Now suppose, for the purpose of deriving a contradiction, that $x^*$ is not an optimal solution of the SOCP.  Since we showed above that $x^*$ is feasible, there then exists a step $\tilde d^*\in\RR^n$ so that $\tilde x = x^* + \tilde d^*$ is feasible for \eqref{eq:socp} and $c^T\tilde d^*<0$.
Then, because $\KKK_j\subseteq\CCC_j(\hat\YYY^{k_t}_j)$, for each $t$, $\tilde d^{k_t}= x^*-x^{k_t}+\tilde d^*$ is feasible for \eqref{eq:cplusgLP}, and because $d^{k_t}$ is an optimal solution of \eqref{eq:cplusgLP}, we have $(c+g^{k_t})^Td^{k_t} \leq (c+g^{k_t})^T\tilde d^{k_t}$.  Taking the limit $t\to\infty$, we obtain $c^Td^{\infty} \leq c^T\tilde d^{*}<0$, where we used $\lim_{t\to\infty}g^{k_t} = \lim_{t\to\infty} H^{k_t}d^{k_t} = 0$, due to the definition of $g^{k_t}$ and \eqref{eq:Hdlimit}.  However, this contradicts \eqref{eq:ctdlim}.  Therefore, $x^*$ must be a solution of the SOCP.
\end{proof}

For later reference, we highlight the limit \eqref{eq:Hdlimit} established in the above proof.
\begin{lemma}\label{cor:Hdlimit}
Suppose that there exists $\rho^{\infty}>0$ so that $\rho^k= \rho^{\infty}>0$ for all large $k$.  Then $\lim_{k\to\infty}H^k d^k=0$.
\end{lemma}

We are now ready to prove that the algorithm is globally convergent under the following standard regularity assumption.
\begin{assumption}
\label{ass:slater}
The SOCP 
is feasible and Slater's constraint qualification holds, i.e., there exists a feasible point $\tilde x\in\RR^n$ and $\epsilon>0$ so that $\tilde x + v$ is feasible for any $v\in\RR^n$ with $\|v\|\leq\epsilon$.
\end{assumption}

This assumption implies that the multiplier estimates are bounded.
\begin{lemma}\label{lem:zbounded}
Suppose that Assumption~\ref{ass:slater} holds. Then $\{\hat z^k\}$ is bounded.
\end{lemma}
\begin{proof}
Let $\tilde x$ and $\epsilon$ be the quantities from Assumption~\ref{ass:slater}.
Note that the QP corresponding to the optimality conditions \eqref{eq:opt_qp} is
\begin{align*}
\min_{d\in\RR^n} \ & c^Td + \tfrac12 d^TH^k d\\
\text{s.t.} \ & A(x^k+d) \leq b,\; x_j^k + d_j \in \CCC_j(\hat\YYY^k_j),\; j\in\JJJ.
\end{align*}
Since $x^{k+1}=x^k+d^k$ when $d^k$ is the step accepted by the algorithm, it follows that $x^{k+1}$ is an optimal solution of the QP
\begin{align*}
O_{\rm primal} = \min_{x\in\RR^n} \ & (c^T-H^kx^k)x + \tfrac12 x^TH^k x\\
\text{s.t.} \ & Ax^{k+1} \leq b,\; x_j^{k+1} \in \CCC_j(\hat\YYY^k_j),\; j\in\JJJ,
\end{align*}
the Lagrangian dual of which is
\begin{subequations}
\begin{align}
O_{\rm dual} = \max_{x,z\in\RR^n,\lam\in\RR^m} \ & -b^T\lam  - \tfrac12 x^TH^k x\\
\text{s.t.} \ & c - H^kx^k + H^kx +A^T\lam -z  =0, \label{eq:dualQP2}\\
& z \in \CCC^\circ_j(\hat\YYY^k_j),\; j\in\JJJ, \quad\lam\geq 0.
\end{align}
\end{subequations}
By \eqref{eq:opt_qp}, $(x^{k+1},\hat\lam^{k}$,$\hat z^k$) is a primal-dual optimal solution of these QPs.

Define $v=-\epsilon\frac{\hat z^k}{\|\hat  z^k\|}$.  Then $\|v\|\leq\epsilon$, and Assumption~\ref{ass:slater} implies that $\tilde x + v\in\KKK_j\subseteq\CCC_j(\hat\YYY_j^k)$. Since $\hat z^k\in\CCC^\circ_j(\hat\YYY_j^k)$, we have with \eqref{eq:dualQP2} that
\begin{equation}\label{eq:vbarx}
0\leq (\tilde x+v)^T\hat z^k = v^T\hat z^k + \tilde x^T( c - H^k\tilde x + H^kx^{k+1} +A^T\hat\lam^k).
\end{equation}
Since $H^k$ is positive definite, it is
\begin{equation}\label{eq:H2sides}
0\leq(\tilde x - x^{k+1})^TH^k(\tilde x-x^{k+1}) = \tilde x^TH^k\tilde x - 2\tilde x^TH^kx^{k+1} + (x^{k+1})^TH^kx^{k+1}.
\end{equation}
Furthermore, Slater's condition implies strong duality, that is
\begin{align}\nonumber
b^T\hat\lam^k + \frac12(x^{k+1})^TH^kx^{k+1} &= -O_{\rm dual} = -O_{\rm primal} \\&=  -(c-H^kx^{k})^Tx^{k+1} - \tfrac12 (x^{k+1})^TH^k x^{k+1}.\label{eq:optsame}
\end{align}
Finally, since $\tilde x$ is feasible for the SOCP, \eqref{eq:socp_2} and $\hat\lam^k\geq0$ imply $\tilde x^T A^T\hat\lam^k\leq b^T\hat\lam^k$.
Subtracting $v^T\hat z^k$ on both sides of \eqref{eq:vbarx}, this, together with \eqref{eq:H2sides} and \eqref{eq:optsame}, yields
\begin{align*}
    \epsilon\|\hat z^k\| &\leq \tilde x^Tc - \frac12 \tilde x^TH^k\tilde x + \frac12 (x^{k+1})^TH^kx^{k+1} + b^T\hat\lam^k \\
    & = \tilde x^Tc - \frac12 \tilde x^TH^k\tilde x - c^Tx^{k+1} - \frac 12(x^{k+1})^TH^kx^{k+1}.
\end{align*}
The first two terms are independent of $k$, and since $X$ is bounded by Assumption~\ref{ass:slater} and $H^k$ is uniformly bounded by \eqref{eq:Hrescale}, we can conclude that $\hat z^k$ is uniformly bounded.
\end{proof}

It is easy to see that the penalty parameter update rule \eqref{eq:rhoupdate} and Lemma~\ref{lem:zbounded} imply the following result.
\begin{lemma}\label{lem:rhoconstant}
Suppose Assumption~\ref{ass:slater} holds. Then there exists $\rho^{\infty}$ and $K_\rho$ so that $\rho^k=\rho^{\infty}> \rho^{\infty}_{\min}$, where 
$\rho^{\infty}_{\min}\geq \rho^{k}_{\min}=\|z^k_{\JJJ,0}\|_\infty$ for all $k\geq K_\rho$.
\end{lemma}

We can now state the main convergence theorem of this section.

\begin{theorem}\label{thm:pdlimit}
Suppose that Assumptions~\ref{ass:boundedness} and \ref{ass:slater} hold. Then Algorithm~\ref{alg:full} either terminates in Step~\ref{s:stop} with an optimal solution, or it generates an infinite sequence of iterates $\{(x^{k+1},\hat\lam^k,\hat z^k)\}_{k=0}^\infty$, each limit point of which is a primal-dual solution of the SOCP~\eqref{eq:socp}.
\end{theorem}
\begin{proof}
\begin{sloppypar}
Let $\{(x^{k_t+1},\hat\lam^{k_t},\hat z^{k_t})\}$ be a subsequence converging to a limit point $(x^*,\lam^*,z^*)$.
No matter whether an iterate is computed from the optimal solution of \eqref{eq:qp_Y}, \eqref{eq:newton_qp}, or \eqref{eq:newton_soc}, the iterates satisfy the optimality conditions \eqref{eq:opt_qp}.
In particular, from \eqref{eq:opt_qp_4} we have for any $j\in\JJJ$ that $\hat z_j^{k_t}\in\CCC_j^\circ(\hat\YYY_j^{k_t})\subseteq\KKK_j$ and $(x_j^{k_t+1})^T\hat z_j^{k_t}=0$.  In the limit, we obtain $z_j^*\in\KKK_j$ (since $\KKK_j$ is closed) and $(x_j^*)^Tz_j^*=0$.
Lemma~\ref{lem:rhoconstant} yields that $\rho^k=\rho^{\infty}$ for all large $k$, and so Lemma~\ref{lem:xlimit} implies that $x^*$ is feasible, i.e., $x_j^*\in\KKK_j$.  Therefore, \eqref{eq:opt_socp_3} holds.
Using Lemma~\ref{cor:Hdlimit} we can take the limit in \eqref{eq:opt_qp_1} and \eqref{eq:opt_qp_2} and deduce also the remaining SOCP optimality conditions \eqref{eq:opt_socp_1} and \eqref{eq:opt_socp_2} hold at the limit point.
\end{sloppypar}
\end{proof}

\begin{remark}\label{rem:infeasible}
In case the SOCP is infeasible, we have two possible outcomes.  Either, Algorithm~\ref{alg:full} terminates in some iterations because one of the QPs is infeasible, or $\lim_{k\to\infty}\rho^k=\infty$ (reverse conclusion of Lemma~\ref{lem:xlimit}).
\end{remark}

\subsection{Identification of extremal-active cones}

We can only expect fast local convergence under some non-degeneracy assumptions.  Throughout this section, we assume that Assumption~\ref{ass:socpnondegen}  holds.
Under this assumption, $(x^*,\lam^*,z^*)$ is the unique optimal solution \cite[Theorem 22]{alizadeh2003second}, and Theorem~\ref{thm:pdlimit} then implies that
\begin{equation*}
\lim_{k\to\infty} (x^{k+1},\hat\lam^k,\hat z^k) = (x^*,\lam^*,z^*).
\end{equation*}

First, we prove a technical result that describes elements in $\CCC^\circ_j(\YYY_j)$ in a compact manner.  For this characterization to hold, condition \eqref{eq:Y0} for the initialization $\YYY_j^0$ of the set of hyperplane-generating points is crucial.
\begin{lemma}\label{lem:zinC}
Let $y_j\in\RR^{n_j}$ with $\bar y_j\neq0$ and $y_{j0}\geq0$. Further, let $\Phi_j(z_j,y_j):=z_{j0}-\|\bar z_j+\bar y_j\|_1-\|\bar y_j\|$. Then the following statements hold for $z_j,y_j\in\RR^{n_j}$:
\begin{enumerate}[(i)]
\item $z_j\in\CCC_j^\circ(\YYY^0_j\cup\{y_j\})$ if $\Phi_j(z_j,y_j)\geq0$.
\label{en:zinC}
\item $z_j\in\interior(\CCC_j^\circ(\YYY^0_j\cup\{y_j\}))$ if $\Phi_j(z_j,y_j)>0$.\label{en:zinintC}
\end{enumerate}

\end{lemma}
\begin{proof}
For (\ref{en:zinC}):  Suppose $\Phi_j(z_j,y_j)\geq0$, then
\begin{equation*}\label{eq:zj0}
z_{j0}\geq \|\bar z_j+\bar y_j\|_1+\|\bar y_j\|.
\end{equation*}
Define $\bar s_j=\bar z_j+\bar y_j$ and choose $\sigma_j^+\in\RR_+^{n_j-1}$ and $\sigma_j^-\in\RR_+^{n_j-1}$ so that $\bar s_{j}=\sigma_{j}^+-\sigma_{j}^-$ and $|s_{ji}|=\sigma_{ji}^++\sigma_{ji}^-$ for all $i=1,\ldots,n_j-1$.
Then we have  
\begin{align*}
z_{j0} & = \sum_{i=1}^{n_j-1}\sigma_{ji}^+ + \sum_{i=1}^{n_j-1}\sigma_{ji}^- + \sigma_j + \eta_j \\
\bar z_{j} & = \bar s_j-\bar y_j =\sigma_{j}^+ - \sigma_{j}^- - \sigma_j\frac{\bar y_{j}}{\|\bar y_j\|}
\end{align*}
with $\sigma_j=\|\bar y_j\|$ and some $\eta_j\in\RR_+$. Using \eqref{eq:na_r}, this can be rewritten as
\[
z_j = -\sum_{i=1}^{n_j-1}\sigma_{ji}^+\na r_j(-e_{ji}) -\sum_{i=1}^{n_j-1}\sigma_{ji}^-\na r_j(e_{ji}) -  \sigma_{j}\na r_j(y_{j}) + \eta_j e_{j0}.
\]
By the definition of $\CCC_j^\circ$ in \eqref{eq:Cdual}, this implies that $z_j\in\CCC_j^\circ(\hat\YYY^0_j\cup\{y_j\})$ where $\hat\YYY_{j}^0$ is defined in \eqref{eq:Y0}.  Since $\hat\YYY_{j}^0\subseteq\YYY_{j}^0$ from \eqref{eq:Y0}, we have $\CCC_j^\circ(\hat\YYY^0_j\cup\{y_j\})\subseteq\CCC_j^\circ(\YYY^0_j\cup\{y_j\})$, and the claim follows.

For (\ref{en:zinintC}): Suppose $\Phi_j(z_j,y_j)>0$.  Because $\Phi_j$ is a continuous function, there exists a neighborhood $N_{\epsilon}(z_j)$ around $z_j$ so that $\Phi_j(\hat z_j,y_j)>0$ for all $\hat z_j\in N_{\epsilon}(z_j)$.  From part (\ref{en:zinC}) we then have $N_{\epsilon}(z_j)\subseteq \CCC_j^\circ(\YYY^0_j\cup\{y_j\})$, and consequently $z_j\in\interior(\CCC_j^\circ(\YYY^0_j\cup\{y_j\}))$.
\end{proof}

\begin{theorem}\label{lem:EEE}
For all $k$ sufficiently large, we have $\EEE(x^k)=\EEE(x^*)$.
\end{theorem}
\begin{proof}
Choose $j\not\in\EEE(x^*)$, then $x^*_j\neq 0$. Because $x_j^k\to x^*_j$, it is $x_j^k\neq0$ or, equivalently, $j\not\in\EEE(x^k)$ for $k$ sufficiently large.
For the remainder of this proof we consider $j\in\EEE(x^*)$ and show that $j\in\EEE(x^k)$ for large $k$.  
Note that strict complementarity in Assumption~\ref{ass:socpnondegen} implies that $z_j^*\in\interior(\KKK_j)$, i.e., $r_j(z_j^*)<0$, and consequently $z_{j0}^*>0$.  

\begin{sloppypar}
First consider the iterations in which fast NLP-SQP steps are accepted in Steps \ref{s:testNewton} or \ref{s:testSOC}.  For the purpose of deriving a contradiction, suppose there exists an infinite subsequence so that $x^{k_t+1}=x^{k_t}+d^{S,k_t}$ or $x^{k_t+1}=x^{k_t}+d^{S,k_t}+s^{k_t}$ and $j\not\in\hat\EEE^{k_t}$.  
Then $j\not\in\hat\EEE^{k_t}$ implies $x^{k_t+1}_{j0}>0$ (according to the termination condition in the while loop in Step~\ref{s:whileEEE}).  We also have $\hat{\YYY}_j^{k_t}=\{\check x_j^{k_t}\}$ where $\check x_j^{k_t}=x_j^{k_t}$ from \eqref{eq:hatYSQP} or $\check x_j^{k_t}=x_j^{k_t}+d_j^{S,k_t}$ from \eqref{eq:hatYSOC}.
Condition \eqref{eq:opt_qp_4} yields $z_j^{k_t}\in\CCC^\circ_j(\{\check x_j^{k_t}\})$, so by \eqref{eq:Cdual} it is $z_j^{k_t}=-\sigma_j\na r_j(\check x_j^{k_t}) + \eta_j e_{j0}$ for some $\sigma_j,\eta_j\geq0$, as well as $x_j^{k_t+1}\in\CCC_j(\{\check x_j^{k_t}\})$, which by \eqref{eq:CCC_j} implies $\na r_j(\check x_j^{k_t})^Tx_j^{k_t+1}\leq0$.  Then complementarity yields
\[
0=(z_j^{k_t})^Tx_j^{k_t+1} = -\sigma_j\na r_j(\check x_j^{k_t})^Tx_j^{k_t+1} + \eta_j x^{k_t+1}_{j0}\geq \eta_j x^{k_t+1}_{j0}.
\]  
Since $x^{k_t+1}_{j0}>0$ and $\eta_j\geq0$, we must have $\eta_j=0$, and consequently $z_j^{k_t}=-\sigma_j\na r_j(\check x_j^{k_t})$.  It is easy to see that $r_j(-\sigma_j\na r_j(\check x_j^{k_t}))=0$. Since $z_j^{k_t}\to z_j^*$, continuity of $r_j$ yields $r_j(z_j^*)=0$, in contradiction to $z_j^*\in\interior(\KKK_j)$.
We thus showed that $j\in\hat\EEE^{k}$ for all large iterations $k$ in which the NLP-SQP step was accepted, and consequently \eqref{eq:hatYSQP} and \eqref{eq:hatYSOC} yield $\hat\YYY_j^{k}=\YYY_j^{k}$ for such $k$.
\end{sloppypar}

In all other iterations \eqref{eq:hatYregular} holds, and overall we obtain
\begin{equation}\label{eq:Ystack}
\YYY_j^0\subseteq\YYY_j^{k}\subseteq\hat\YYY_j^{k} \text{ for all sufficiently large } k.
\end{equation}
Let us first consider the case when $\bar z^*_j=0$.  Then $\|\bar z_j^*\|-z^*_{j0}=r_j(z_j^*)<0$ yields $z_{j0}^*>0$.  
To apply Lemma~\ref{lem:zinC} choose any $i\in\{1,\ldots,n_j-1\}$ and let $y_j=e_{ji}$.  Then $\|y_j\|_1=\|y_j\|=1$ and $\Phi_j(z_j^*,y_j)=z^*_{j0}>0$.  Since $\hat z_j^k\to z_j^*$ and $\Phi_j$ is continuous, $\Phi_j(\hat z_j^k,y_j)>0$ for sufficiently large $k$, and by Lemma~\ref{lem:zinC}, $\hat z^k_j\in\interior(\CCC_j^\circ(\YYY^0_j\cup\{y_j\}))$.  Since $y_j\in\YYY^0_j$ and \eqref{eq:Ystack} holds, we also have $\hat z^k_j\in\interior(\CCC_j^\circ(\hat\YYY^k_j))$.
General conic complementarity in \eqref{eq:opt_qp_4} then implies that $x_j^{k+1}=x_j^k+d_j^k=0$ for all large $k$, or equivalently, $j\in\EEE(x^k)$ for $k$ sufficiently large, as desired.

Now consider the case $\bar z^*_j\neq0$. 
For the purpose of deriving a contradiction, suppose there exists a subsequence $\{x^{k_t}\}_{t=0}^\infty$ so that $j\not\in\EEE(x^{k_t})$, i.e., $x^{k_t}\neq0$,
for all $t$.  
Because $\check z^k_j\to z^*_j$, $\bar z^*_j\neq0$, and $r_j(z_j^*)<0$, we may assume without loss of generality that $r_j(\check z_j^{k_t})<0$ and $\bar {\check z}_j^{k_t}\neq 0$ for all $t$.
Using this and $x_j^{k_t}\neq 0$, we see that the update rule \eqref{eq:YYYplus_du} in Step~\ref{s:updateYdual} adds $-\check z^{k_t}_j$ to $\YYY_j^{k_t+1}$.  With \eqref{eq:Ystack}, we have
\begin{equation}\label{eq:checkzinY}
-\check z^{k_t}_j\in\YYY_j^{k_t+1}\subseteq\YYY_j^{k_{t+1}}\subseteq\hat\YYY_j^{k_{t+1}} \text{ for all } t.
\end{equation}

Recall the mapping $\Phi_j$ defined in Lemma~\ref{lem:zinC}
and note that $\Phi_j(z^*_j,-z^*_j)=z^*_{j0}-\|\bar z^*_j\|=-r_j(z^*_j)>0$.
Since both $\hat z^k_j$ and $\check z^k_j$ converge to $z_j^*$ and $\Phi_j$ is continuous, it is $\Phi_j(\hat z_j^{k_{t+1}-1}, -\check z^{k_t}_j)>0$ for all large $t$, and therefore, by Lemma~\ref{lem:zinC}, $\hat z_j^{k_{t+1}-1}\in\interior(\CCC_j^\circ(\YYY^0_j\cup\{-\check z^{k_t}_j)\}))\stackrel{\eqref{eq:checkzinY}}{\subseteq}\interior(\CCC_j^\circ(\hat\YYY^{k_{t+1}-1}_j))$ for all large $t$.  Conic complementarity in \eqref{eq:opt_qp_4} then implies that $x_j^{k_{t+1}} =x_j^{k_{t+1}-1}+d_j^{k_{t+1}-1}=0$.  This is a contradiction of the definition of the subsequence 
$\{x^{k_t}\}_{t=0}^\infty$.  
\end{proof}


\begin{remark}\label{rem:zstartinterior}
In the proof of Theorem~\ref{lem:EEE}, we saw that $\Phi_j(z^*_j,-z^*_j)>0$ if $j\in\EEE(x^*)$ and $\bar z_j^*\neq0$.  Since $\Phi_j$ is continuous, this implies that there exists a neighborhood $N_\epsilon(z^*_j)$ so that 
$\Phi_j(z_j,-y_j)>0$, and consequently $z_j\in\interior(\CCC_j^\circ(\YYY^0_j\cup\{-y_j\}))$, for all $z_j,y_j\in N_\epsilon(z^*_j)$.
\end{remark}


\subsection{Quadratic local convergence}

As discussed in Section~\ref{s:reformulation}, since $x^*$ is a solution of the SOCP \eqref{eq:socp}, it is also a solution of the nonlinear problem \eqref{eq:equiv_nlp}.  We now show that Algorithm~\ref{alg:full} eventually generates steps that are identical to SQP steps for \eqref{eq:equiv_nlp}.  Then Theorem~\ref{thm:sqp_soc} implies that the iterates converge locally at a quadratic rate.

We first need to establish that the assumptions for Theorem~\ref{thm:sqp_soc} hold.
\begin{lemma}\label{lem:nondegen}
Suppose that Assumption~\ref{ass:socpnondegen} holds for the SOCP \eqref{eq:socp}.  Then Assumption~\ref{ass:sqp} holds for the NLP~\eqref{eq:equiv_nlp}.
\end{lemma}
\begin{proof}
Let $\lam^*$ and $z^*$ be the optimal multipliers for the SOCP corresponding to $x^*$, satisfying \eqref{eq:opt_socp}.  Assumption~\ref{ass:socpnondegen} implies that $\lam^*$ and $z^*$ are unique \cite[Theorem 22]{alizadeh2003second}.

Let $j\in\DDD(x^*)$ and define $\mu^*_j = z^*_{j0}\geq0$.  
If $0=r_j(x^*_j)=x^*_{j0}-\|\bar x^*_{j}\|$, complementarity \eqref{eq:opt_socp_3} implies, for all $i\in\{1,\ldots n_j\}$, that
$
0=x^*_{j0}z^*_{ji} + x^*_{ji}z^*_{j0} = \|\bar x^*_{j}\|z^*_{ji} + x^*_{ji}z^*_{j0},
$, or equivalently, $z^*_{ji} = -z^*_{j0} \frac{x^*_{ji}}{\|\bar x^*_{j}\|}$; see \cite[Lemma 15]{alizadeh2003second}. Using \eqref{eq:na_r}, this can be written as
\begin{equation}\label{eq:zismu}
z^*_j = -z^*_{j0}\na r_j(x^*_j) =  -\mu^*_{j}\na r_j(x^*_j).
\end{equation}
On the other hand, if $r_j(x^*_j)<0$, i.e., the constraint \eqref{eq:equiv_nlp_3} is inactive, then $x^*_j\in\interior(\KKK_j)$ and complementarity \eqref{eq:opt_socp_3} yields $z_j^*=0$ (see \cite[Definition 23]{alizadeh2003second}) and therefore $\mu_j^*=0$. Consequently, \eqref{eq:zismu} is also valid in that case.
Finally, we define $\nu^*_j=z^*_j$ for all $i\in\EEE(x^*)$.
With these definitions, \eqref{eq:opt_socp_1} can be restated as
\begin{equation}\label{eq:nlp_stationary}
c + A^T \lambda^* + \sum_{j\in \DDD(x^*)} \mu^*_j \na r_j(x^*) - \nu^* = 0,
\end{equation}
where $\nu^*\in\RR^n$ is the vector with the values of $\nu_j^*$ at the components corresponding to $j\in\EEE(x^*)$ and zero otherwise.
We now prove parts (\ref{ass:strictcompl}), (\ref{ass:licq}), and (\ref{ass:sosc}) of Assumption~\ref{ass:sqp}.

Proof of (\ref{ass:strictcompl}): Let $j\in\DDD(x^*_j)$. We already established that $r_j(x_j^*)<0$ yields $\mu^*_j=0$.  Now suppose that $r_j(x^*_j)=0$.  Then $x_j^*\in\bd(\KKK_j)\setminus\{0\}$.  Since strict complementarity is assumed, we have $z_j^*\in\bd(\KKK_j)\setminus\{0\}$ (see the comment after Assumption~\ref{ass:socpnondegen}), which in turn yields $z_j^*\neq 0$ and hence $\mu_j^*\neq0$.

Proof of (\ref{ass:licq}): 
Since we need to prove linear independence only of those constraints that are active at $x^*$, we consider only those rows $A_\AAA$ of $A$ for which \eqref{eq:equiv_nlp_1} is binding.

Without loss of generality suppose $x^*$ is partitioned into four parts, $(x^*)^T=((x_\BBB^*)^T\ (x_\III^*)^T\ (x_\EEE^*)^T\  (x_\FFF^*)^T)$, where $x_\BBB^*$, $x_\III^*$, and $x_\EEE^*$ correspond to the variables in the cones $\BBB=\{j\in\JJJ:r_j(x^*_j)=0, x_j^*\neq0\}$, $\III=\{j\in\JJJ:r_j(x^*_j)<0\}$, and $\EEE=\EEE(x^*)$, respectively, and $x^*_\FFF$ includes all components of $x^*$ that are not in any of the cones.
Further suppose that $(x_\BBB^*)^T=((x^*_1)^T \ \ldots (x^*_{p_\BBB})^T)$, where $\BBB=\{1,\ldots,p_{\BBB}\}$, and that $A_\AAA$ is partitioned in the same way.

Primal non-degeneracy of the SOCP implies all that matrices of the form 
\[
\begin{pmatrix}
[A_\AAA]_1 & \cdots & [A_\AAA]_{p_\BBB} &  [A_\AAA]_\III & [A_\AAA]_\EEE & [A_\AAA]_\FFF\\
\alpha_1 \na r_1(x_1^*)^T& \cdots &  \alpha_{p_\BBB} \na r_{p_\BBB}(x_{p_\BBB}^*)^T & 0^T & v^T & 0^T
\end{pmatrix}
\]
have linear independent rows for all scalars $\alpha_i$ and vectors $v$, not all zero \cite[Eq.~(50)]{alizadeh2003second}. 
This implies that the rows of $A_\AAA$, together with the gradient of any one of the binding constraints in \eqref{eq:equiv_nlp_3} and \eqref{eq:equiv_nlp_4} are linearly independent.  Because the constraint gradients, which are of the form $\na r_j(x^*_j)$ and $e_{ij}$, share no nonzero components when extended to the full space, we conclude that the gradients of all active constraints are linearly independent at $x^*$, i.e., the LICQ holds.

Proof of (\ref{ass:sosc}):
For the purpose of deriving a contradiction, suppose that there exists a direction $d\in\RR^n\setminus\{0\}$ that lies in the null space of the constraints of \eqref{eq:equiv_nlp} that are binding at $x^*$ and for which $d^TH^*d\leq0$. 

\begin{sloppypar}
Since $d$ is in the null space of the binding constraints, we have $A_{\AAA}d=0$, $\na r_j(x^*)^Td=0$ for $j\in\BBB$, and $d_j=0$ for all $j\in\EEE$.  Premultiplying \eqref{eq:nlp_stationary} by $d^T $ gives \end{sloppypar}

\begin{equation}\label{eq:ctd0}
0=c^T d + (\lam^*)^T\underbrace{  A_{\AAA}d}_0  + \sum_{j\in \BBB}\mu^*\underbrace{ \na r_j(x^*)^T d}_0 + \sum_{j\in \III}\underbrace{\mu^*}_0 \na r_j(x^*)^T d + \underbrace{(\nu^*)^T d}_0 = c^Td. 
\end{equation}
What remains to show is that $d$ is a feasible direction for the SOCP, i.e., there exists $\beta>0$ so that $x^*+\beta d$ is feasible for the SOCP.  Because of \eqref{eq:ctd0}, this point has the same objective value as $x^*$ and is therefore also an optimal solution of the SOCP.  This contradicts the fact that Assumption~\ref{ass:socpnondegen} implies that the optimal solution is unique \cite[Theorem 22]{alizadeh2003second}.

By the definition of $H^*$ in Assumption~\ref{ass:sqp} and the choice of $d$, we have
\[
0\geq d^TH^*d = \sum_{j\in\DDD(x^*)}\mu^*_j d_j^T\na^2 r_j(x^*_j)d_j = \sum_{j\in\BBB}\mu^*_j d_j^T\na^2 r_j(x^*_j)d_j .
\]
Since for all $j\in\BBB$, the Hessian $\na^2 r_j(x^*_j)$ is positive semi-definite and $\mu^*_j>0$ from Part (i), this yields $d_j^T\na^2 r_j(x^*_j)d_j = 0$ for all $j\in\BBB$.

Let $j\in\BBB$.  Then from \eqref{eq:rHess}
\begin{equation}\label{eq:dhrd}
0=d_j^T \na^2 r_j(x^*) d_j = \frac{\norm{\bar d_j}^2\norm{\bar x^*_j}^2-(\bar d_j^T \bar x^*_j)^2}{\norm{\bar x^*_j}^3}.
\end{equation}
The definition of $\BBB$ implies $r_j(x_j^*)=0$ and so $x^*_{j0}=\|\bar x^*_j\|$. Since $d_j$ is in the null space of $\na r_j(x_j^*)$, we have $0=\na r_j(x^*_j)^T d_j = -d_{j0} + \frac{\bar d_j^T \bar x_j}{\|\bar x^*_j\|}$, which in turn yields $d_{j0}x^*_{j0} = \bar d_j^T \bar x^*_j$.  Finally, using these relationships together with \eqref{eq:dhrd} gives
\[
0 = \norm{\bar d_j}^2\norm{\bar x^*_j}^2-(\bar d_j^T \bar x^*_j)^2 = \|\bar d_j\|^2\, (x^*_{j0})^2 - (d_{j0}x^*_{j0})^2
\]
and so $d_{j0}^2 = \|\bar d_j\|^2$.  All of these facts imply that for any $\beta\in\RR$,
\begin{align*}
& \|\bar x^*_j + \beta \bar d_j\|^2 - (x^*_{j0} + \beta d_{j0})^2 \\ =& \|\bar x^*_j\|^2 + 2 \beta \bar d_j^T \bar x^*_j + \beta^2 \|\bar d_j\|^2 - \left((x^*_{j0})^2 + 2\beta d_{j0}x^*_{j0} +\beta^2 d_{j0}^2\right)=0,
\end{align*}
which implies $r_j(x^*_j + \beta d_j)=0$ and therefore $x^*_j + \beta d_j\in\KKK_j$.

Further, because $d$ lies in the null space of the active constraints, we have, for any $\beta\in\RR$, that $x^*_j+\beta d_j=0\in\KKK_j$ for all $j\in\EEE(x^*)$ and $A_\AAA(x^*+\beta d)=b_\AAA$.  Finally, since $r_j(x^*_j)<0$ and hence $x^*_j\in\interior(\KKK_j)$ for all $j\in\JJJ\setminus(\EEE(x^*)\cup\BBB)$, and since $x^*_j$ is strictly feasible for all non-binding constraints in \eqref{eq:socp_2}, there exists $\beta>0$ so that $x^*+\beta d$ satisfies all constraints in \eqref{eq:socp}.
\end{proof}

\begin{theorem}
\begin{sloppypar}
Suppose that $c_H>\|H^*\|$. Then the primal-dual iterates $(x^{k+1},\hat\lam^k,\hat z^k)$ converge locally to $(x^*,\lam^*,z^*)$ at a quadratic rate.
\end{sloppypar}
\end{theorem}
\begin{proof}
We already established in Theorem~\ref{thm:pdlimit} that the iterates converge to the optimal solution, and since $H^k\to H^*$ and $c_H>\|H^*\|$, the Hessian is not rescaled according to \eqref{eq:Hrescale} in Step~\ref{s:calcH}.
Using Theorem~\ref{lem:EEE} we know that, once $k$ is sufficiently large, the step  $d^{S,k}$ computed in Step~\ref{s:newton} of Algorithm~\ref{alg:full} is identical with the SQP step from \eqref{eq:sqp_subproblem} for \eqref{eq:equiv_nlp}; we can ignore \eqref{eq:newton_qp4} here because $x^*_{j0}>0$ and $d^{S,k}_{j0}\to 0$ and therefore this constraints is not active for large $k$.  This also implies that the condition in Step~\ref{s:whileEEE} is never true and thus $\hat\EEE_k=\EEE(x^*)$.  If the decrease condition in Step~\ref{s:testNewton} is not satisfied, by a similar argument we have that $s^k$ computed in Step~\ref{s:soc} is the second-order correction step from \eqref{eq:sqp_soc} for \eqref{eq:equiv_nlp}.  
Due to Lemma~\ref{lem:nondegen} we can now apply Theorem~\ref{thm:sqp_soc} to conclude that either $d^{S,k}$ or $d^{S,k}+s^k$ is accepted to define the next iterate for large $k$ and that the iterates converge at a quadratic rate.
\end{proof}

\section{Numerical Experiments}\label{sec:numerical}
In this section, we examine the performance of Algorithm~\ref{alg:full}.
First, using randomly generated instances, we consider three types of starting points: (i) uninformative default starting point (cold start), (ii) solution of a perturbed instance, (iii) solution computed by an interior-point SOCP solver whose accuracy we wish to improve.
Then we briefly report results using the test library CBLIB.
The numerical experiments were performed on an Ubuntu 22.04 Linux server with a 2.1GHz Xeon Gold 5128 R CPU and 256GB of RAM.

%

\subsection{Implementation}

We implemented Algorithm~\ref{alg:full} in MATLAB R2021b, with parameters $c_{\text{dec}} = 10^{-6}$, $c_{\text{inc}} = 2$, $c_H=10^{12}$,  and $\rho^{-1} = 50$.  
In each iteration, we identify $\EEE(x^k)=\{j\in\JJJ:\norm{x^k_j}_{\infty}<10^{-6}\}$ and $\DDD(x^k)=\{j\in\JJJ \setminus\EEE(x^k): \norm{\bar x^k_j}>10^{-8}\}$.
The set $\YYY_j^0$ is initialized to $\hat\YYY_j^0$ (see \eqref{eq:Y0}),
and $\lam^0$ is a given starting value for $\lam$, if provided, and zero otherwise.
In addition, since the identification of the optimal extremal-active set $\EEE(x^*)$ requires $z^*_j\in\CCC_j^\circ(\YYY_j)$, we add $-\check z^0_j$ to $\YYY^0_j$, where  $\check z^0= c+A^T \lambda^0$. 

The algorithm terminates when the violation of the SOCP optimality conditions \eqref{eq:opt_socp} for the current iterate satisfies
\begin{equation}\label{eq:SOCPopterror}
E(x^k,\lam^k,\check z^k) = \max\left \{  \begin{array}{l}
     \|[Ax^k-b]^+\|_\infty,  \|(Ax^k-b)\circ \lam^k\|_\infty , \|[-\lambda^k]^+\|_\infty \\[1ex]
     \max_{j\in \setJ} \left\{[r_j(x^k)]^+, [r_j(\check z^k)]^+ , \abs{(x_j^k)^T \check z^k_j} \right\}
\end{array}\right\}\leq \epsilon_{\text{tol}}
\end{equation}
with $\check z^k = c+A^T\lam^k$,
for some $\epsilon_{\text{tol}}>0$. 
%
 
%
As in \cite{ipopt}, the sufficient descent condition \eqref{eq:armijo} is slightly relaxed by
\[
\varphi(\hat x^{k+1};\rho^k) - \varphi(x^k;\rho^k)- 10\epsilon_{\text{mach}}\abs{\varphi(x^k; \rho^k)}
\leq c_{\text{dec}}\left(m^k(x^k+d;\rho^k) - m^k(x^k;\rho^k)\right)
\]
to account for cancellation error, where $\epsilon_{\text{mach}}$ is the machine precision.
Finally, to avoid accumulating very similar hyperplanes that would lead to degenerate QPs, we do not add a new generating point $v_j$ to $\YYY_j^k$ if there already exists $y_j\in \YYY_j^k$ such that $\left\|\frac{\bar v_j}{\|\bar v_{j}\|}-\frac{\bar y_j}{\|\bar y_j\|}\right\|_\infty \leq 10^{-10}.$

In these experiments, we disabled the second-order correction step (Steps \ref{s:soc}--\ref{s:exit_soc}) because we noticed that it was never accepted in practice. 
In a more sophisticated implementation, one would include a heuristic that attempts to detect the Maratos effect and then triggers the second-order correction step in specific situations. 

The QPs were solved using ILOG CPLEX V12.10, with optimality and feasibility tolerances set to $10^{-9}$ and ``dependency checker'' and ``numerical precision emphasis'' enabled, using the primal simplex method.  When CPLEX did not report a solution status ``optimal'' and the QP KKT error was above  $10^{-9}$, a small perturbation was added to the Hessian matrix, i.e., we replaced $H^k$ by $H^k+10^{-7}\cdot I$.  This helped in some cases in which CPLEX (incorrectly) reported that $H^k$ was not positive semi-definite.
%
If CPLEX still did not find a QP solution with KKT error less than $10^{-9}$, we attempted to resolve the QP with the barrier method, the dual simplex method, and the primal simplex method again, until one was able to compute a solution. 
 If all solvers failed for QP \eqref{eq:newton_qp}, the algorithm continued in Step~\ref{s:begininner}.  If no solver was able to solve \eqref{eq:qp_Y}, we terminated the main algorithm and declared a failure.

We emphasize that the purpose of our implementation is to assess whether the proposed algorithm exhibits behavior that validates the stated goals: Convergence from any starting point and rapid local convergence to highly accurate solutions.
In its current implementation, it requires more computation time than highly sophisticated commercial solvers such as MOSEK or CPLEX, which were developed over decades and have highly specialized linear algebra routines that are tightly integrated into the algorithms.
As we observed at the end of Section~\ref{s:complete}, many of the QPs in Algorithm~\ref{alg:full} that are solved in succession are similar to each other, and savings in computation times should therefore be achievable. 
However, our prototype implementation based on the Matlab CPLEX interface does not allow us to utilize callback functions for adding or removing hyperplanes.  Achieving these savings in computation time thus requires a more sophisticated implementation, a task that is outside of the scope of this paper.  Consequently, we do not report solution times here.

\subsection{Randomly generated QCQPs}
The experiments were performed on randomly generated SOCP instances of varying sizes, specified by $(n,m,K)$. Here, $n,m\geq1$ are the number of variables and linear constraints, respectively.  $K\geq1$ specifies the number of cones of each ``activity type'':
$|\EEE(x^*)|=K$, $|\{j\in\JJJ:r_j(x^*_j)=0, x_j^*\neq0\}|=K$, and $|\{j\in\JJJ:r_j(x^*_j)<0\}|=K$, i.e., there are $K$ cones that are extremal-active, $K$ that are active at the boundary, and $K$ that are inactive at the optimal solution $x^*$.  The dimensions of the cones are randomly chosen.
In addition, there are variables that are not part of any cone, with bounds chosen in a way so that the non-degeneracy assumption,  Assumption~\ref{ass:socpnondegen}, holds.
A detailed description of the problem generation is stated in Appendix A in
\cite{techreport}.

\begin{table}
\centering
\begin{tabular}{|r|r|r|c|c|c|c|c|}
\hline
\multicolumn{1}{|c|}{$n$}    & \multicolumn{1}{c|}{$m$}        & \multicolumn{1}{c|}{$K$}      &  solved &  \multicolumn{1}{c|}{total}      & SQP &  total & total \\
& & & & iter & iter  & QP \eqref{eq:newton_qp} & QP \eqref{eq:qp_Y}\\ \hline
200 	&60 	&10	&30	&6.67	&6.67	&9.77    &0.00\\
400 	&120	&20	&30	&7.20	&7.20	&11.57   &0.00\\
1000	&300	&50	&30	&7.23	&7.23	&12.17   &0.00\\\hline
200 	&60 	&4	&30	&7.53	&7.07	&11.83   &0.90\\
400 	&120	&8	&30	&8.27	&7.77	&14.20   &1.00\\
1000	&300	&20	&30	&8.67	&7.80	&15.93   &1.83\\\hline
200 	&60 	&2	&30	&8.47	&7.87	&13.90   &1.20\\
400 	&120	&4	&30	&8.87	&8.07	&15.30   &1.60\\
1000	&300	&10	&30	&9.47	&8.43	&17.27   &1.97\\\hline
\end{tabular}
\caption{Results with $x^0 = 0$, $\epsilon_{\text{tol}} = 10^{-7}$, average per-size statistics taken over 30 random instances. ``solved'': number of instances solved (out of 30); ``total iter'': total number of iterations in Algorithm~\ref{alg:full}; ``SQP iter'': number of iteration in which NLP-SQP step was accepted in Steps~\ref{s:testNewton} or \ref{s:testSOC}; ``total QP \eqref{eq:newton_qp}'' / ``total QP \eqref{eq:qp_Y}'': Total number of QPs of that type solved. }
\label{table:0_initial}
\end{table}
Table \ref{table:0_initial} summarizes the performance of the algorithm with an uninformative starting point $x^0=0$.
Each row lists average statistics for a given problem size ($n,m,K$), taken over 30 random instances.
We see that the proposed algorithm is very reliable and solved every instance to the tolerance $\epsilon=10^{-7}$.
The average number of iterations is mostly between 7--9, during most of which the second-order NLP-SQP step was accepted.

To give an idea of the computational effort, we report the number of times QPs~\eqref{eq:newton_qp} and \eqref{eq:qp_Y} were solved.
And we can draw further conclusions from this data:
Consider, for example, the last row.  At the beginning of each iteration, QP~\eqref{eq:newton_qp} is solved to obtain the NLP-SQP step. The difference with the total number of iterations, i.e., 17.27-9.47=7.80, gives us the total number of times in which the guess $\hat\EEE^k$ of the extremal-active cones needed to be corrected in Steps~\ref{s:whileEEE}--\ref{e:whileEEE}.  In other words, on average, the loop Steps~\ref{s:whileEEE}--\ref{e:whileEEE} is executed 7.80/9.47=0.82 times per iteration.
Similarly, the last column tells us the total number of iterations of the loop in Steps~\ref{s:begininner}--\ref{s:endinner}.  The loop was only executed when the NLP-SQP step was not accepted, so in 9.47-8.43=1.04 iterations, taking 1.97/1.04=1.89 loop iterations on average.

The experiments are presented in three groups where the ratio between $n$ and $K$ is kept constant.  As the number of cones, $K$, decreases from one group to the next, the average size of the individual cones increases by a factor of $2.5$ and $2$, respectively.
This increase seems to result in slightly more iterations in which the SQP step was rejected, indicating that the simple linearization \eqref{eq:newton_qp3} of the non-extremal-active cones becomes sometimes insufficiently accurate.

 \begin{table}
 \centering
\begin{tabular}{|r|r|r|c|c|c|c|c|c|c|}
\hline
\multicolumn{1}{|c|}{$n$}    & \multicolumn{1}{c|}{$m$}        & \multicolumn{1}{c|}{$K$}      &  \multicolumn{1}{c|}{total}      & SQP &  total & total  & Mosek & final\\
& & & iter & iter  & QP \eqref{eq:newton_qp} & QP \eqref{eq:qp_Y}& error & error\\ \hline
              200 &              60 &              10 &            1.10 &                    1.07 &                    1.10 &                    0.07&        2.33e-06 &        1.63e-10 \\
              400 &             120 &              20 &            1.03 &                    1.00 &                    1.03 &                    0.03&        2.67e-06 &        1.70e-10 \\
             1000 &             300 &              50 &            1.07 &                    1.03 &                    1.07 &                    0.03&        3.49e-06 &        1.76e-10 \\ \hline
              200 &              60 &               4 &            1.03 &                    1.03 &                    1.03 &                    0.00&        5.97e-06 &        1.69e-10 \\
              400 &             120 &               8 &            1.00 &                    1.00 &                    1.00 &                    0.00&        2.28e-06 &        1.87e-10 \\
             1000 &             300 &              20 &            1.03 &                    0.83 &                    1.03 &                    0.27&        5.20e-06 &        1.72e-10 \\ \hline
              200 &              60 &               2 &            1.00 &                    1.00 &                    1.00 &                    0.00&        2.02e-06 &        1.53e-10 \\
              400 &             120 &               4 &            1.13 &                    1.10 &                    1.13 &                    0.03&        4.85e-06 &        2.03e-10 \\
             1000 &             300 &              10 &            1.20 &                    1.10 &                    1.20 &                    0.13&        1.22e-05 &        2.41e-10 \\ \hline

\end{tabular}
\caption{Result with MOSEK solution as $x^0$, $\epsilon_{\text{tol}} = 10^{-9}$. All instances were solved. ``Mosek error'': Optimality error $E$ \eqref{eq:SOCPopterror} at Mosek solution; ``final error'': Optimality error $E$ at final iterate of Algorithm~\ref{alg:full}.}
\label{table:mosek}
\end{table}

The remaining experiments in this section investigate to which degree the algorithm is able to achieve our primary goal of taking advantage of a good starting point.
We begin with an extreme situation, in which we first solve an instance with the interior-point SOCP solver MOSEK V9.1.9 (called via CVX), using the setting \texttt{cvx\_precision=high}  corresponding to the MOSEK tolerance $\epsilon=\epsilon_{\rm mach}^{2/3}$, and give the resulting primal-dual solution as starting point to Algorithm~\ref{alg:full}.  Choosing any tighter MOSEK tolerances leads to failures in several problems.
Table~\ref{table:mosek} summarizes the results.
In all cases, the algorithm converges rapidly to an improved solution, reducing the error by 4 orders of magnitude, most of the time with only a single second-order iteration.   The Mosek error was dominated by the violation of complementarity.
This demonstrates the ability of the proposed method to improve the accuracy of a solution computed by an interior-point method.
%

\begin{table}
\centering
\begin{tabular}{|r|r|r|c|r|c|c|c|}
\hline
\multicolumn{1}{|c|}{$n$}    & \multicolumn{1}{c|}{$m$}        & \multicolumn{1}{c|}{$K$}      &  solved &  \multicolumn{1}{c|}{total}      & SQP &  total & total \\
& & & & iter & iter  & QP \eqref{eq:newton_qp} & QP \eqref{eq:qp_Y}\\ \hline
200	&60	&10	&30	&1.00	&0.97	&1.00	&0.07\\
400	&120	&20	&30	&1.00	&0.97	&1.00	&0.03\\
1000	&300	&50	&30	&1.00	&0.97	&1.00	&0.07\\\hline
200	&60	&4	&30	&1.00	&1.00	&1.00	&0.00\\
400	&120	&8	&30	&1.00	&0.93	&1.00	&0.07\\
1000	&300	&20	&30	&1.00	&0.87	&1.00	&0.20\\\hline
200	&60	&2	&30	&1.00	&1.00	&1.00	&0.00\\
400	&120	&4	&30	&1.00	&0.97	&1.00	&0.03\\
1000	&300	&10	&30	&1.07	&1.00	&1.03	&0.07\\\hline
\end{tabular}
\caption{Result with $10^{-3}$ perturbation, $\epsilon_{\text{tol}} = 10^{-7}$. }
\label{table:perturb3}
\end{table}
\begin{table}
\centering

\begin{tabular}{|r|r|r|c|r|c|c|c|}
\hline
\multicolumn{1}{|c|}{$n$}    & \multicolumn{1}{c|}{$m$}        & \multicolumn{1}{c|}{$K$}      &  solved &  \multicolumn{1}{c|}{total}      & SQP &  total & total \\
& & & & iter & iter  & QP \eqref{eq:newton_qp} & QP \eqref{eq:qp_Y}\\ \hline
200	&60	&10	&30	&1.20	&1.07	&1.00	&0.19\\
400	&120	&20	&30	&1.33	&1.17	&1.00	&0.73\\
1000	&300	&50	&30	&1.60	&1.23	&1.02	&1.29\\\hline
200	&60	&4	&30	&1.27	&1.13	&1.02	&0.71\\
400	&120	&8	&30	&1.67	&1.27	&1.16	&0.48\\
1000	&300	&20	&30	&2.10	&1.40	&1.27	&0.57\\\hline
200	&60	&2	&30	&1.67	&1.33	&1.24	&0.42\\
400	&120	&4	&30	&2.30	&1.87	&1.30	&0.38\\
1000	&300	&10	&30	&3.67	&2.53	&1.53	&0.59\\\hline
    \end{tabular}
\caption{Result with $10^{-1}$ perturbation, $\epsilon_{\text{tol}} = 10^{-7}$.}
\label{table:perturb1}
\end{table}

For the final experiments, summarized in Tables \ref{table:perturb3} and \ref{table:perturb1}, the starting point is the MOSEK solution of a perturbed problem, in which 10\% of the objective coefficients $c$ were perturbed by uniformly distributed random noise of the order of $10^{-3}$ and $10^{-1}$, respectively.
For the small perturbation, similar to Table~\ref{table:mosek}, Algorithm~\ref{alg:full} terminated in one iteration most of the time.
More iterations were required for the larger perturbation, but still significantly fewer compared to the uninformative starting point, see Table~\ref{table:0_initial}.

\subsection{CBLIB instances}

\begin{table}
\centering
\begin{tabular}{|l|r|r|r|c|r|r|c|}
\hline
\multicolumn{1}{|c|}{Problem} & \multicolumn{1}{|c|}{\# var} & \multicolumn{1}{|c|}{\# con} & \multicolumn{1}{|c|}{\# soc} & solved/ & total & SQP  &  iter warm/ \\
\multicolumn{1}{|c|}{subset}  &  &  &  &  total & \multicolumn{1}{|c|}{iter} &  \multicolumn{1}{|c|}{iter} & iter cold \\\hline
\texttt{             10/20*} &   3024.0 &   6030.6 &    336.0 &    10/11 &      6.5 &      6.5 &     0.57\\
\texttt{     achtziger\_*}   &   2314.0 &   1856.7 &    257.0 &     6/7 &     21.3 &     10.2 &     0.53\\
\texttt{            as\_*}   &   4909.0 &   5695.7 &    669.0 &    20/20 &     10.4 &      8.4 &     0.37\\
\texttt{            ck\_*}   &   2376.0 &   2375.0 &     15.0 &    90/90 &      5.8 &      4.6 &     0.40\\
\texttt{     classical\_*}   &    159.3 &    200.4 &      1.0 &   409/409 &      7.3 &      7.3 &     0.49\\
\texttt{           clay*}   &    457.2 &    691.9 &     82.9 &    11/12 &     67.6 &     13.2 &     1.31\\
\texttt{         estein*}   &    103.1 &    165.2 &     14.0 &     9/9 &     18.3 &      4.4 &     0.30\\
\texttt{           flay*}   &    175.0 &    408.0 &      4.0 &    10/10 &      7.0 &      3.6 &     0.34\\
\texttt{             fo[7-9]*} &    208.6 &    540.2 &     15.9 &    19/19 &     12.3 &      3.0 &     0.53\\
\texttt{             m[3-9]*}  &    150.2 &    370.5 &     12.8 &     8/8 &     10.9 &      4.0 &     0.44\\
\texttt{             nb*}      &   3098.8 &    321.0 &    816.0 &     4/4 &      4.8 &      4.2 &     3.43\\
\texttt{         netmod*}     &    989.7 &   2593.0 &      4.7 &     3/3 &      4.0 &      4.0 &     0.33\\
\texttt{            no7\_*}   &    169.0 &    438.0 &     14.0 &     5/5 &     10.6 &      4.4 &     0.46\\
\texttt{          o[7-9]\_*} &    183.2 &    466.9 &     14.7 &     9/9 &      8.7 &      3.2 &     0.66\\
\texttt{            pp-*}     &   2960.0 &   2221.0 &    370.0 &     6/6 &     12.8 &     12.8 &     0.17\\
\texttt{        robust\_*}   &    253.9 &    298.3 &      2.0 &   420/420 &      6.8 &      6.8 &     0.67\\
\texttt{         sched\_*} &   7361.5 &   3685.0 &      1.5 &     4/4 &     12.5 &     10.0 &     0.35\\
\texttt{     shortfall\_*} &    249.9 &    294.3 &      2.0 &   420/420 &      7.8 &      7.8 &     0.37\\
\texttt{           slay*} &    393.0 &    936.0 &     14.0 &    14/14 &      8.6 &      7.0 &     0.27\\
\texttt{          sssd-*} &    286.0 &    314.5 &     18.0 &    16/16 &      7.6 &      5.6 &     0.42\\
\texttt{      stolpe07-*} &   1483.0 &   1202.0 &    164.7 &     3/3 &     29.0 &     18.3 &     0.71\\
\texttt{            tls*} &    521.4 &    989.8 &     51.6 &     5/6 &      7.6 &      7.0 &     0.49\\
\texttt{        turbine*} &    201.6 &    289.3 &     52.4 &     7/7 &      4.9 &      4.9 &     0.30\\
\texttt{       uflquad-*} &   8011.0 &   6811.0 &   1333.3 &     3/4 &     15.3 &     14.7 &     0.19\\
\texttt{        wiener\_*} &   2297.5 &   2508.5 &     71.3 &    41/45 &     12.4 &      9.9 &     0.43\\
\texttt{           misc} &   1534.1 &   1374.5 &    158.9 &    11/14 &     10.5 &      7.6 &     0.33\\\hline
\end{tabular}

\caption{Results for CBLIB instances, averaged per problem group, $\epsilon_{\text{tol}} = 10^{-5}$.  ``Problem subset'': name of problem group; ``\# var'': number of variables; ``\# con'': number of linear constraints; ``\# soc'': number of second-order cone constraints; ``solved/total'': number of solved vs.\ total instances; ``total iter'': number of iterations in Algorithm~\ref{alg:full}; ``SQP iter'': number of iterations in which NLP-SQP step was accepted;  ``iter warm/iter cold'': iterations for warm start divided by iterations for cold start (only for instances solved in both settings).}
\label{table:cblib}
\end{table}

To demonstrate the robustness of the algorithm we also solved instances from the Conic Benchmark Library CBLIB \cite{cblib}.  Some instances involve rotated second-order cone constraints, and we reformulated them so that they fit into our standard form \eqref{eq:socp}. We chose all 1,575 instances with at most 10,000 variables and 10,000 constraints.  Integer variables were relaxed to be continuous.

Table~\ref{table:cblib} shows the statistics for the starting point $x^0=0$, where instances were grouped into sets of problems with similar names, and the remaining ones are collected in \texttt{misc}.  The method was able to solve 99.2\% of the instances, where  10 problems could not be solved due to failures of the QP subproblem solver, and  Algorithm~\ref{alg:full} exceeded the maximum number of 200 iterations in 2 cases.  In comparison, MOSEK, with default settings, failed on 5 instances (those were solved correctly with Algorithm~\ref{alg:full}), incorrectly declared 3 instances to be infeasible, and labeled 6 instances to be unbounded (of which 3 were solved by Algorithm~\ref{alg:full}). 
We observed that some instances, especially those in the \texttt{clay*}, \texttt{fo[7-9]*}, \texttt{m[3-9]*}, \texttt{no7*}, \texttt{o[7-9]\_*} subsets, are degenerate, having an optimal objective function value of 0, and the assumption necessary to prove fast local convergence is violated.
This matches the observations in the table, where the SQP step was accepted only in a relatively small fraction of the iterations.
%

To showcase the warm-starting feature of the algorithm, we took the 1,563 previously successfully solved instances, perturbed 10\% of the entries of the final primal-dual iterate by a random perturbation, uniformly chosen in $[-0.1,0.1]$, and used this as the starting point for a warm-started run.  Here, QP subproblem failure occurred in 3 cases and 2 instances exceeded the iteration limit.  As we can see from the last column in Table~\ref{table:cblib}, on average, the number of iterations was reduced in most cases, often by more than 50\%.

\section{Concluding remarks}\label{sec:conclusion}
We presented an SQP algorithm for solving SOCPs and proved that it converges from any starting point and achieves local quadratic convergence for non-degenerate SOCPs.
Our numerical experiments indicate that the algorithm is reliable, converges quickly when a good starting point is available, and produces more accurate solutions than a state-of-the-art interior-point solver.
%
%

Future research would investigate whether the proposed algorithm is a valuable alternative for interior-point methods for small problems or for the solution of a sequence of related SOCPs.  An efficient implementation of the algorithm beyond our Matlab prototype would be tightly coupled with a tailored active-set QP solver that efficiently adds or removes cuts instead of solving each QP essentially from scratch. Parametric active-set solvers such as qpOASES \cite{qpoases} or QORE \cite{qore} might be suitable options since they do not require primal or dual feasible starting points. 

\section*{Acknowledgments}
We thank Javier Pe\~na for his suggestion for the proof of Lemma~\ref{lem:zbounded}, as well as three  referees whose comments helped us to improve the paper.

\bibliographystyle{abbrv}
\bibliography{reference}

\newpage
\appendix
\section{Generation of random instances}
\label{app:problems}




The test instances used for the numerical experiments have the form
\begin{align*}
\minimize{x\in \RR^n}& \sum_{j=1}^p c_j^T x_j + c_{p+1}^T x_{p+1}\\
\subject & \sum_{j = 1}^p A_j x_j + A_{p+1} x_{p+1} = b,\\
         & x_j\in \setK_j, \quad j\in [p],\\
         & x_{j0} \leq 1000, \quad 0\leq x_{p+1} \leq 1000, 
\end{align*}
where $(c_1, \cdots c_p, c_{p+1})\in \RR^{n_1}\times \cdots \times \RR^{n_p}\times \RR^{n_{p+1}}$ is the partition of the  objective vector $c$,  and $(A_1, \cdots, A_{p+1})$ is the partition of the column vectors of constraint matrix $A\in \RR^{m\times n}$. 
The subvector $x_{p+1}$ includes all optimization variables that are not in any of the cones.

In contrast to the original SOCP \eqref{eq:socp}, this formulation includes linear equality constraints, but it is straightforward to generalize Algorithm \ref{alg:full} and its convergence proofs for this setting.
The formulation above also includes large upper bounds on all variables so that the set $X$ defined in \eqref{eq:setX} is bounded.  However, the upper bounds were chosen so large that they are not active at the optimal solution.
\begin{algorithm}
    \caption{Random Instance Generation}\label{alg:data_gen}
    \begin{algorithmic}[1]
        \Require { $n$ (number of variables), $m$ (number of linear constraints), $p$ (total number of cones), $K_0$ (number of extremal-active cones), $K_I$ (number of inactive cones), $K_B$ (number of cones active at the boundary, excluding extremal-active cones), $d$ (density of nonzeros in constraint matrix)}.
        Conditions: $n,m,p,K_0,K_I,K_B\in\mathbb{Z}_+$, $d\in(0,1]$, $p=K_0+K_I+K_B$, $n>2p$, $n>m+K_B+\lfloor(m+K_B)/2\rfloor$.
        \State Randomly choose positive integers $n_{1},\ldots,n_{K_0}$ so that $n_j\geq2$ and $\sum_{j=1}^{K_0} n_j = \lfloor(m+K_B)/2\rfloor$. 
        \State Randomly choose positive integers $n_{K_0+1},\ldots,n_{p}$ so that $n_j\geq2$ and $\sum_{j=K_0+1}^{p} n_j = m+K_B$. \label{s:exp_start}
        \For{$j = 1, \ldots, K_0$} \Comment{Generate $x_j^*=0$, $z_j^*\in \interior (\setK_j)$}
        \State Set $x_j^*\gets 0$. 
        \State Sample $\bar z_j\sim (-10, 10)^{n_j-1}$, $z_{j0}^*\sim (1, 5)$ and $\epsilon_j \sim (0,1)$. Set $\bar z_j^*\gets (2+\epsilon)\frac{z^*_{j0}}{\|\bar z_j\|} \bar z_j$.
        \EndFor
        \For{$j = K_0+1,\ldots, K_0+K_I$} \Comment{Generate $z_j^*=0$, $x_j^*\in \interior (\setK_j)$}
        \State Set $z_j^*\gets 0$.
        \State Sample $\bar x_j\sim (-10, 10)^{n_j-1}$, $x_{j0}^*\sim (1, 5)$ and $\epsilon_j \sim (0,1)$. Set $\bar x_j^*\gets (2+\epsilon)\frac{x^*_{j0}}{\|\bar x_j\|} \bar x_j$.
        \EndFor
        \For{$j = K_0+K_I+1,\ldots, p$}\Comment{Generate $z_j^*, x_j^*\in \bd(\setK_j)$}
        \State Sample $\bar x_j\sim (-10, 10)^{n_j-1}$ and  $x_{j0}^*\sim (1, 5)$. Set $\bar x_j^*\gets \frac{x^*_{j0}}{\|\bar x_j\|} \bar x_j$.
        \State Sample $\beta\sim(1, 5)$. Set $\bar z^*_j\gets -\beta \bar x_j^*$ and $\bar z^*_{j0} \gets \beta x_{j0}^*$
        \EndFor
        \State Set $n_{\rm fix}\gets n-m-K_B-\sum_{j=1}^{K_0} n_j$ and $n_{\rm free}\gets n-n_{\rm fix}-\sum_{j=1}^{p} n_j$. \label{s:nfixed} 
        \For{$j = 1,\ldots,n_{\rm fix}$} 
        \State Set [$x_{p+1}^*]_{j}\gets 0$ and sample $[z^*_{p+1}]_j\sim (1, 5)$.
        \EndFor
        \For{$j = n_{\rm fix}+1,\ldots,n_{\rm free}$} 
        \State Set [$z_{p+1}^*]_{j}\gets 0$ and sample $[x^*_{p+1}]_j\sim (1, 5).$
        \EndFor
        \State Sample linear independent rows $A_i\sim (-5, 5)^{n}$ with density $d$, for $i=1,\cdots,m$. 
        \State Call \texttt{pcond} to calculate primal condition number $c_p$ and call \texttt{dcond} to calculate dual condition number $c_d$. \label{s:pdcond}
        \State \textbf{if} {$c_p>10^{5}$ or $c_d>10^5$} \textbf{then }go to Step \ref{s:exp_start} and redo the process.
        \State Sample $\lambda^*\sim (1, 10)^{m}$. Set $b\gets A x^*$ and $c\gets -A^T\lambda^* + z^*$.
        \State \Return $A, b, c, x^*, \lambda^*, z^*$.
    \end{algorithmic}
\end{algorithm}

Algorithm~\ref{alg:data_gen} describes how the data for these instances are generated.
The algorithm generates an optimal primal-dual solution in a way so that, at the solution, $K_0$ cones are extremal-active, for $K_I$ cones the optimal solution lies in the interior of the cone, and for $K_B$ the optimal solution is at the boundary but not the extreme point of the cone.  In our experiments, there is an equal number of $K$ of each type.
The number of variables that are active at the lower bound, $n_{\rm fix}$, is chosen in a way so that the optimal solution is non-degenerate.
The linear constraints reduce the number of degrees of freedom by $m$, each extremal active cone by $n_j$, and each otherwise active cone by $1$.  
Step~\ref{s:nfixed} calculates the number of constraints that are active at zero so that the total number of degrees of freedom fixed by constraints equals $n$.
Lastly, in Step \ref{s:pdcond}, we use \texttt{pcond} and \texttt{dcond} functions provided in SDPPack~\cite{nayakakuppam1997sdppack} to double-check if a generated instance is also numerically non-degenerate. These functions return the primal and dual conditions numbers, $c_p$ and $c_d$, respectively, and we discard an instance if either number is above a threshold.
Only about 1\% of instances created were excluded in this manner.

\end{document}